\documentclass[a4, 12pt, leqno]{article}
\usepackage{amsmath}
\usepackage{amssymb}
\usepackage{amsthm}
\newtheorem{df}{Definition}[section]
\newtheorem{pr}[df]{Proposition}
\newtheorem{Th}[df]{Theorem}

\newtheorem{lm}[df]{Lemma}
\newtheorem{Cr}[df]{Corollary}
\newtheorem{as}[df]{Assumption}

\begin{document}

\title{\bf The Hasimoto Transformation for a 
Finite Length Vortex Filament and its Application}
\author{Masashi A{\sc iki}\({}^{\ast} \)}
\date{}
\maketitle
\vspace*{-0.5cm}

\begin{abstract}
We consider two nonlinear equations, the Localized Induction Equation 
and the cubic nonlinear Schr\"odinger Equation, and prove that the solvability 
of certain initial-boundary value problems for each equation is equivalent
through the generalized Hasimoto transformation. As an application, we prove the 
orbital stability of plane wave solutions of the nonlinear Schr\"odinger equation
based on stability estimates obtained for the 
Localized Induction Equation by the author in a paper in preparation.
As far as the author knows, this is the first time that the 
analysis of the Localized Induction Equation, along with the Hasimoto transformation, 
provided new insight for the nonlinear Schr\"odinger equation.
\end{abstract}

\vspace*{1cm}
\noindent
\({}^{\ast} \)Department of Mathematics\\
\ Faculty of Science and Technology, Tokyo University of Science\\
\ 2641 Yamazaki, Noda, Chiba 278-8510, Japan\\
\ E-mail: aiki\verb|_|masashi\verb|@|ma.noda.tus.ac.jp \\
\ ORCID: 0000-0002-3476-6894

\section{Introduction and Problem Setting}

A vortex filament is a space curve on which the vorticity of the fluid 
is concentrated. 
Vortex filaments are used to model very thin vortex structures such as vortices that 
trail off airplane wings or propellers. 
The model equation we consider in this paper is the 
Localized Induction Equation (LIE) given by
\begin{align*}
\mbox{\mathversion{bold}$x$}_{t}=\mbox{\mathversion{bold}$x$}_{s}\times
\mbox{\mathversion{bold}$x$}_{ss}
\end{align*}
where \( \mbox{\mathversion{bold}$x$}(s, t)= {}^{t}(x_{1}(s ,t), x_{2}(s ,t), x_{3}(s ,t))\) is the 
position vector of the vortex filament parametrized by 
its arc length \( s \) at time \( t\), 
\( \times \) is the exterior product in the three-dimensional Euclidean space,
and subscripts \( s\) and \(t \) are differentiations with the respective variables.

The LIE, which is derived by applying the localized induction approximation to the 
Biot--Savart integral, was first derived by 
Da Rios \cite{20} in 1906 and was re-derived twice independently by 
Murakami et al. \cite{22} in 1937 and by 
Arms and Hama \cite{21} in 1965. Since then, many researches have been done on 
the LIE and many results have been obtained.
Nishiyama and Tani \cite{5,13} proved the unique solvability of initial and 
initial-boundary value problems in Sobolev spaces. 
The author \cite{23} and the author and Iguchi \cite{11} proved the unique solvability 
of initial-boundary value problems in Sobolev spaces with different boundary 
conditions. 
In the above papers, the equation for the tangent vector,
\( \mbox{\mathversion{bold}$v$}:=\mbox{\mathversion{bold}$x$}_{s}\), given by
\begin{align}
\mbox{\mathversion{bold}$v$}_{t}=\mbox{\mathversion{bold}$v$}\times 
\mbox{\mathversion{bold}$v$}_{ss}
\label{vfe}
\end{align}
is introduced in the analysis and plays an important role.
Equation (\ref{vfe}) is sometimes called the Vortex Filament Equation (VFE).

Koiso \cite{12} considered a geometrically generalized setting in which he
rigorously proved the 
equivalence of the solvability of the initial value problem for the 
VFE and the cubic nonlinear Schr\"odinger equation. 
This equivalence was first shown by 
Hasimoto \cite{14} in which he studied the formation of solitons 
on a vortex filament. He defined a transformation of variable known as the 
Hasimoto transformation to transform the VFE into a
nonlinear Schr\"odinger equation. 
The Hasimoto transformation is a change of variable given by 
\[ q(s,t) = \kappa(s,t) \exp \left( {\rm i} \int ^{s}_{0} \tau(r,t) \,{\rm d}r
 \right), \]
where \( {\rm i}\) is the imaginary unit, 
\( \kappa \) is the curvature, and \( \tau \) is the torsion of the filament.
Defined as such, it is well known that \( q \) satisfies the 
nonlinear Schr\"odinger equation given by 
\begin{eqnarray}
{\rm i}q_{t}=q_{ss}+\frac{1}{2}\lvert q\rvert^{2}q.
\label{NLS}
\end{eqnarray}
The original transformation proposed by Hasimoto uses the torsion of the filament 
in its definition, which means that the transformation is undefined at points 
where the curvature of the filament is zero. 
Koiso \cite{12} constructed a transformation, sometimes referred to as the 
generalized Hasimoto transformation, and gave a mathematically rigorous 
proof of the equivalence of the VFE and (\ref{NLS}). 
More precisely, Koiso proved that the solvability of initial value problems 
for the VFE and (\ref{NLS}) are equivalent.
More recently, 
Banica and Vega \cite{15,16,25} and Guti\'errez, Rivas, 
and Vega \cite{17} utilized the generalized Hasimoto transformation 
to construct and analyze
a family of self-similar solutions of the LIE which forms a corner in finite time. 
In Chang, Shatah, and Uhlenbeck \cite{26} and 
Nahmod, Shatah, Vega, and Zeng \cite{24}, they considered the 
Schr\"odinger maps and employed a Hasimoto type transformation to 
prove the correspondence between the solution of the Schr\"odinger maps and the 
solution of a nonlinear Schr\"odinger type equation. The Schr\"odinger maps
are a generalization of the Heisenberg model for a ferromagnetic spin system given by
\begin{align*}
\mbox{\mathversion{bold}$m$}_{t}=
\mbox{\mathversion{bold}$m$}\times \Delta \mbox{\mathversion{bold}$m$},
\end{align*}
where \( \Delta \) is the Laplacian in \( \mathbf{R}^{n}\) and the 
unknown variable \( \mbox{\mathversion{bold}$m$}\) takes values in 
\( \mathbf{S}^{2}\). The Heisenberg model can be seen as a multi-dimensional 
version of the
VFE.

All of the results mentioned above which utilizes the Hasimoto transformation
consider either an infinitely long filament defined on the whole line or 
a closed filament defined on the torus.
As far as the author knows, the rigorous justification and application of the 
Hasimoto transformation for problems describing the 
motion of filaments with end-points have not been done.
In this paper, we justify the Hasimoto transformation and prove the equivalence of the 
solvability of the following initial-boundary value problems.
\begin{align}
\left\{
\begin{array}{ll}
\mbox{\mathversion{bold}$v$}_{t} =\mbox{\mathversion{bold}$v$}\times 
\mbox{\mathversion{bold}$v$}_{ss}, & s\in I_{L}, \ t>0, \\[3mm]
\mbox{\mathversion{bold}$v$}(s,0)=\mbox{\mathversion{bold}$v$}_{0}
, & s\in I_{L}, \ t>0, \\[3mm]
\mbox{\mathversion{bold}$v$}(0,t)=\mbox{\mathversion{bold}$e$}_{1}, \ 
\mbox{\mathversion{bold}$v$}(L,t)=\mbox{\mathversion{bold}$b$}, & t>0,
\end{array}\right.
\label{slant}
\end{align}
\begin{align}
\left\{
\begin{array}{ll}
\displaystyle
{\rm i}q_{t}=q_{ss}+\frac{1}{2}\lvert q\rvert^{2}q, & s\in I_{L},t>0, \\[3mm]
q(s,0)=q_{0}(s), & s\in I_{L}, \\[3mm]
q_{s}(0,t)=q_{s}(L,t)=0, & t>0. 
\end{array}\right.
\label{plane}
\end{align}
Problem (\ref{slant}) is an initial-boundary value problem for the VFE
which describes the 
motion of a vortex filament on a slanted plane, considered in a previous paper by the 
author \cite{23}. Here, \( L>0\) is the length of the initial filament, 
\( I_{L}\subset \mathbf{R}\) is the interval \( (0,L)\), 
\( \mbox{\mathversion{bold}$e$}_{1}={}^{t}(1,0,0)\), and
\( \mbox{\mathversion{bold}$b$}\in \mathbf{R}^{3}\)
is an arbitrary constant vector with unit length. The boundary datum at 
\( s=0\) was chosen as \( \mbox{\mathversion{bold}$e$}_{1}\) without loss of 
generality, because the VFE is invariant under rotation.
Problem (\ref{plane}) is an initial-boundary value problem for the focusing cubic 
nonlinear Schr\"odinger equation, where \( q=q(s,t)\) is a complex-valued function.

As an application of this equivalence, we will prove the orbital stability 
of the plane wave solution
\( q_{R}\)
of (\ref{plane}) given in the form
\begin{align*}
q_{R}(t)=-\frac{1}{R}\exp\left\{ -\frac{{\rm i}t}{2R^{2}} \right\}
\end{align*}
for \( R> L/\pi\) in the Sobolev space \( H^{2}(I_{L})\). This will be done by 
considering problem (\ref{slant}) with appropriate 
initial and boundary data, utilizing energy estimates for the solution
\(\mbox{\mathversion{bold}$v$}\) of 
(\ref{slant}) obtained by the author in \cite{23,27}, and transferring  
the estimates for \( \mbox{\mathversion{bold}$v$}\) into estimates for 
solutions of (\ref{plane}).

As far as the author knows, the results of this paper is the first time
the generalized Hasimoto transformation is utilized to give new insight on 
the nonlinear Schr\"odinger equation from known facts about the
VFE. All of the preceding works utilizing the generalized 
Hasimoto transformation did so to analyze the solution of the VFE
utilizing known facts about the nonlinear Schr\"odinger equation.

The initial value problem for equation (\ref{NLS}) on the torus, 
explicitly given by
\begin{align}
\left\{
\begin{array}{ll}
\displaystyle
{\rm i}u_{t}=u_{ss}+\frac{1}{2}\lvert u\rvert^{2}u, & s\in \mathbf{T},t>0, \\[3mm]
u(s,0)=u_{0}(s), & s\in \mathbf{T},
\end{array}\right.
\label{pNLS}
\end{align}
where \( \mathbf{T}=\mathbf{R}/[-L,L]\), is closely related to problem
(\ref{plane}) since the solvability of problem (\ref{plane}) can be reduced to
the solvability of problem (\ref{pNLS}) by reflection and periodic extension.

It is known by Zakharov and Shabat \cite{40} that 
equation (\ref{NLS}) is completely integrable, and
the solution to problem 
(\ref{pNLS}) posses infinitely many conserved quantities.
This in turn implies that solutions in the Sobolev space 
\( H^{m}\) for \( m\in \mathbf{N}\) 
is bounded in \( H^{m}\) for all time. The solvability of problem (\ref{pNLS}) in 
Lebesgue or Sobolev spaces are known, for example, by
Bourgain \cite{41}. 
Hence, it is natural to ask if particular types of solutions are stable 
in Sobolev spaces.
Namely, the stability of plane wave solutions of equation (\ref{NLS}) has 
been studied by many researchers. 

Zhidkov \cite{28} gives a detailed analysis of the plane wave solutions for 
the initial value problem on the whole space \( \mathbf{R}\).
The stability of plane wave solutions and periodic wave solutions for 
problem (\ref{pNLS}) was investigated by
Rowlands \cite{44}, 
Gallay and H\u{a}r\u{a}gu\c{s} \cite{29,30},
Faou, Gauckler, and Lubich \cite{31}, and 
Wilson \cite{42}.

In particular, Gallay and H\u{a}r\u{a}gu\c{s} \cite{29,30} considered 
problem (\ref{pNLS}) and
proved the orbital stability of periodic wave solutions in \( H^{1}_{per}\),
i.e., the orbital stability in \( H^{1}\) with the perturbation restricted to 
periodic perturbations with the same period as the periodic wave solution.
By definition, periodic wave solutions include plane wave solutions and hence,
the stability results in \cite{29,30} are valid for plane wave solutions as well.

Faou, Gauckler, and Lubich \cite{31} considered the
initial value problem on the torus with general dimension \( d \geq 1 \),
and proved the long-time orbital stability of plane wave solutions
in \( H^{r} \). By long-time they mean 
stability up to time of order \( O(\frac{1}{\varepsilon ^N}) \) where
\( \varepsilon >0 \) is the size of the perturbation in \( H^{r}\) and 
\( N\in \mathbf{N}\).
The index \( r>0 \) must be chosen sufficiently large, depending on \( N\) and 
the \( L^{2}\) norm of the perturbation.

Wilson \cite{42} considered the stability problem of plane wave solutions
for the nonlinear Schr\"odinger equation with general power nonlinearities,
and obtained similar results as \cite{31}.

Other results on the 
initial--boundary value problems for the nonlinear Schr\"odinger equation have been
obtained by Holmer \cite{51}, 
Fokas and Its \cite{52}, Fokas, Its, and Sung \cite{54},
Lenells and Fokas \cite{55,56}, 
Bona, Sun, and Zhang \cite{57}, 
and
Fokas, Himonas, and Mantzavinos \cite{53}.
These results prove the well-posedness of 
initial-boundary value problems under various boundary conditions as well as
obtain representation formulas for boundary values which represent 
unknown boundary values of a solution by known boundary values, but 
do not address the problem of stability of specific solutions.

In summary, we see that up until this paper, the stability of 
plane wave solutions in higher order Sobolev spaces is 
only partially known. Specifically, the regularity for which 
the stability is proved in \cite{31,42} is not given explicitly, and
the time-span for which the stability holds is not global.
On the other hand, the results in this paper give the time-global orbital stability 
of plane wave solutions in Sobolev space \( H^2 \).
This is possible greatly due to the fact that the method of the proof 
given in this paper is vastly different from the preceding works.

The methods utilized to prove the stability of solutions in 
\cite{28,44,29,30,31,42} can be broadly categorized into two types.
Variational methods and methods utilizing the Hamiltonian structure of the
equation.
The variational approach was utilized in 
Cazenave and Lions \cite{62} to prove orbital stability of 
standing waves for the nonlinear Schr\"odinger equation.
The approach utilizing the Hamiltonian structure of the equation was
introduced by Grillakis, Shatah, and Strauss \cite{60,61} for a broad range of 
equations having a Hamiltonian structure.
These methods are widely adopted to approach stability problems for a 
wide variety of dispersive equations.
In contrast, the method employed in this paper is tailor-made specifically
for problem (\ref{slant}) and (\ref{plane}). Hence, 
we are able to obtain more information for our
specific problem, but the method is not as widely applicable to other problems
compared to traditional methods.

The contents of the rest of the paper are as follows.
In Section 2, we introduce basic notations and define 
compatibility conditions related to problems (\ref{slant}) and 
(\ref{plane}). Then, we give a brief explanation of the 
Hasimoto transformation and state our main theorems.
In Section 3, we prove our main theorems. 
We first prove that the solvability of problems
(\ref{slant}) and (\ref{plane}) is equivalent
through the generalized Hasimoto transformation.
We further prove that plane wave solutions of problem 
(\ref{plane}) correspond to a particular type of 
solution, which we call arc-shaped solutions, of
problem (\ref{slant}).

Then, we prove stability estimates for arc-shaped solutions 
of problem (\ref{slant}) and also prove that 
these stability estimates can be transferred to 
stability estimates for plane wave solutions of problem (\ref{plane})
through the generalized Hasimoto transformation.
The stability estimates for arc-shaped solutions are essentially derived 
from standard energy estimates for the perturbation, 
which is much more simple than traditional methods.


\section{Function Spaces, Notations, and Main Theorem}
\setcounter{equation}{0}

We introduce some function spaces that will be used throughout this paper, 
and notations associated with the spaces.
For a non-negative integer \( m\) and \( 1\leq p \leq \infty \), \( W^{m,p}(I_{L})\) 
is the Sobolev space 
containing all real-valued functions that have derivatives in the sense of
 distribution up to order \( m\) 
belonging to \( L^{p}(I_{L})\).
We set \( H^{m}(I_{L}) := W^{m,2}(I_{L}) \) as the 
Sobolev space equipped with the usual inner product, and 
set \( H^{1}_{0}(I_{L}) \) as the closure, with 
respect to the \( H^{1}\)-norm, of the set of smooth functions with compact support.
The norm in \( H^{m}(I_{L}) \) is denoted by \( \| \cdot \|_{m} \) and we 
simply write \( \| \cdot \| \) 
for \( \|\cdot \|_{0} \). Otherwise, for a Banach space \( X\), 
the norm in \( X\) is written as \( \| \cdot \| _{X}\).
The inner product in \( L^{2}(I_{L})\) is denoted by \( (\cdot ,\cdot )\).

For \( 0<T \leq \infty \) and a Banach space \( X\), 
\( C^{m}([0,T];X) \)
( \( C^{m}\big( [0,\infty);X\big)\) when \( T= \infty \)),
denotes the space of functions that are \( m\) times continuously differentiable 
in \( t\) with respect to the norm of \( X\).
The space \( L^{\infty}\big( 0,\infty; X \big) \) denotes the space of functions 
that are essentially bounded in \( t\) with respect to the norm of \( X\)

For any function space described above, we say that a vector valued function 
belongs to the function space 
if each of its components does, and the same for complex-valued functions if 
both the real and imaginary parts do.

Finally, vectors \( \mbox{\mathversion{bold}$e$}_{j}\in \mathbf{R}^{3}\) 
for \( j=1,2,3\) denote the standard basis of \( \mathbf{R}^{3}\).
In other words, \( \mbox{\mathversion{bold}$e$}_{1}={}^{t}(1,0,0)\),
\( \mbox{\mathversion{bold}$e$}_{2}={}^{t}(0,1,0)\), and
\( \mbox{\mathversion{bold}$e$}_{3}={}^{t}(0,0,1) \).
Additionally, we denote the unit sphere in \( \mathbf{R}^{3}\) by
\( \mathbf{S}^{2}\).

\bigskip
We next introduce some definitions in order to state the main theorems of this paper.
First we define the compatibility conditions for both 
(\ref{slant}) and (\ref{plane}).
%
%
%

%
%
%
%
\subsection{Compatibility Conditions for (\ref{slant}) and (\ref{plane})}
First we define the compatibility conditions needed in this paper for
problem (\ref{slant}).
\begin{df}
For \( m=0 \ \text{or} \ 1\), 
we say that 
\( \mbox{\mathversion{bold}$v$}_{0}\in H^{2m+1}(I_{L}) \) and 
\( \mbox{\mathversion{bold}$b$}\in \mathbf{S}^{2}\) satisfy the \( m\)-th order compatibility condition
for {\rm (\ref{slant})} if 
\begin{align*}
\mbox{\mathversion{bold}$v$}_{0}(0)=\mbox{\mathversion{bold}$e$}_{1}, \quad 
\mbox{\mathversion{bold}$v$}_{0}(L)=\mbox{\mathversion{bold}$b$},
\end{align*}
when \( m=0\), and 
\begin{align*}
\mbox{\mathversion{bold}$v$}_{0}(0)\times \mbox{\mathversion{bold}$v$}_{0ss}(0)
=
\mbox{\mathversion{bold}$v$}_{0}(L)\times \mbox{\mathversion{bold}$v$}_{0ss}(L)
=
\mbox{\mathversion{bold}$0$}
\end{align*}
when \( m=1 \).
We also say that 
\( \mbox{\mathversion{bold}$v$}_{0}\) and 
\( \mbox{\mathversion{bold}$b$} \) satisfy the 
compatibility conditions for {\rm (\ref{slant})} up to order \( 1\) if
\( \mbox{\mathversion{bold}$v$}_{0}\) and 
\( \mbox{\mathversion{bold}$b$}\) satisfy both the \( 0\)-th order and the 
\( 1\)-st order compatibility condition for {\rm (\ref{slant})}.
\end{df}
Next we define the corresponding compatibility condition for 
(\ref{plane}).
\begin{df}
For \( q_{0}\in H^{2}(I_{L})\), we say that \( q_{0}\) satisfies the 
\( 0\)-th order compatibility condition for {\rm (\ref{plane})} if
\begin{align*}
q_{0s}(0)=q_{0s}(L)=0
\end{align*}
are satisfied.
\end{df}
%
%
%
%
%
%
\subsection{The Hasimoto Transformation and the Main Theorems}
To state our main theorem, we give a brief explanation of the Hasimoto transformation.
The Hasimoto transformation is a map that relates the solution of 
equation (\ref{vfe}) to the solution of equation (\ref{NLS}) proposed by
Hasimoto \cite{14}. The original transformation proposed by Hasimoto 
isn't always well-defined because the transformation is defined using the 
torsion of the filament, which is not defined at points 
where the curvature of the filament is zero. Later, Koiso \cite{12} 
proved that the solvability of the initial value problem on the torus for 
(\ref{vfe}) and (\ref{NLS}) is equivalent. Koiso did so by constructing a modified
transformation, to which we refer to as the generalized Hasimoto transformation, 
which maps solutions of equation (\ref{vfe}) to solutions of equation (\ref{NLS}). 
This transformation is invertible, and hence, the solvability is equivalent.

One of the aims of this paper is to prove that the generalized 
Hasimoto transformation given by Koiso \cite{12}
can be further modified to prove the equivalence of the solvability of 
problem (\ref{slant}) and problem (\ref{plane}).
More precisely, we prove the following.
\begin{Th}
For \( q_{0}\in H^{2}(I_{L}) \) satisfying the \( 0\)-th order 
compatibility condition for
{\rm (\ref{plane})},
there exists \( \mbox{\mathversion{bold}$v$}_{0}\in H^{3}(I_{L})\) and 
\( \mbox{\mathversion{bold}$b$}\in \mathbf{S}^{2}\)
satisfying \( \lvert \mbox{\mathversion{bold}$v$}_{0}\rvert \equiv 1\)
and the compatibility conditions for {\rm (\ref{slant})}
up to order \( 1\)
such that the following holds. The solution 
\( \mbox{\mathversion{bold}$v$}\in C\big( [0,\infty);H^{3}(I_{L})\big) \cap C^{1}\big( [0,\infty);H^{1}(I_{L})\big) \) of problem 
{\rm (\ref{slant})} with initial datum
\( \mbox{\mathversion{bold}$v$}_{0}\) and boundary datum 
\( \mbox{\mathversion{bold}$b$}\)
corresponds to the 
solution \( q\in C\big( [0,\infty);H^{2}(I_{L})\big) \cap C^{1}\big( [0,\infty);L^{2}(I_{L})\big) \) of problem {\rm (\ref{plane})}
 with initial datum \( q_{0}\) through the 
generalized Hasimoto transformation.

\label{hasi1}
\end{Th}
\begin{Th}
For \( \mbox{\mathversion{bold}$v$}_{0}\in H^{3}(I_{L})\) and 
\( \mbox{\mathversion{bold}$b$}\in \mathbf{S}^{2} \)
satisfying \( \lvert \mbox{\mathversion{bold}$v$}_{0}\rvert \equiv 1 \)
and the compatibility conditions for {\rm (\ref{slant})}
up to order \( 1\),
there exists \( q_{0}\in H^{2}(I_{L}) \) satisfying the \( 0\)-th order 
compatibility condition for {\rm (\ref{plane})}
such that the following holds. The solution 
\( q\in C\big( [0,\infty);H^{2}(I_{L})\big) \cap C^{1}\big( [0,\infty);L^{2}(I_{L})\big) \) of problem 
{\rm (\ref{plane})} with initial datum \( q_{0}\)
corresponds to the solution
\( \mbox{\mathversion{bold}$v$}\in C\big( [0,\infty);H^{3}(I_{L})\big) \cap C^{1}\big( [0,\infty);H^{1}(I_{L})\big) \) of problem 
{\rm (\ref{slant})} with initial datum
\( \mbox{\mathversion{bold}$v$}_{0}\) and boundary datum 
\( \mbox{\mathversion{bold}$b$}\)
through the inverse generalized Hasimoto transformation.
\label{hasi2}
\end{Th}
Theorem \ref{hasi1} and \ref{hasi2} together states that the solvability of 
problem (\ref{slant}) and problem (\ref{plane}) are equivalent in 
suitable Sobolev spaces.
The two theorems also imply that the compatibility conditions for
(\ref{slant}) and (\ref{plane}) correspond to each other 
through the generalized Hasimoto transformation.
Recall from the introduction that the solvability of 
problem (\ref{slant}) and problem (\ref{plane}) are already known.
Hence, the solvability of either problem in itself is not new,
but the fact that the solvability of the two problems are equivalent is new.

As an application of the above two theorems, we prove the 
following theorem.
\begin{Th}
The plane wave solution \( q_{R}\) of problem {\rm (\ref{plane})} given by
\begin{align*}
q_{R}(t)=-\frac{1}{R}\exp \left\{ -\frac{{\rm i}t}{2R^2} \right\}
\end{align*}
with \( \displaystyle R>\frac{L}{\pi}\) is orbitally stable in
\( H^{2}(I_{L}) \). More specifically, for any \( \varepsilon >0\), there exists a 
\( \delta >0\) such that for any \( \phi_{0}\in H^{2}(I_{L})\) satisfying the 
\( 0\)-th order compatibility condition for {\rm (\ref{plane})} and 
\( \| \phi_{0} \|_{2} \leq \delta \), the solution 
\( q\in C\big( [0,\infty);H^{2}(I_{L})\big) \cap C^{1}\big( [0,\infty);L^{2}(I_{L})\big) \) of problem 
{\rm (\ref{plane})} with initial datum
\( q_{0}=q_{R}(0)+\phi_{0}\) satisfies
\begin{align*}
\sup_{t\geq 0} \inf_{\theta \in \mathbf{R}} 
\| \exp({\rm i}\theta)q(t)-q_{R}(t)\|_{2} < \varepsilon.
\end{align*}
\label{stability}
\end{Th}
\noindent
From here on, we will refer to the maps 
\( q_{0}\mapsto \mbox{\mathversion{bold}$v$}_{0}\) and 
\( \mbox{\mathversion{bold}$v$} \mapsto q \) implied in 
Theorem \ref{hasi1} as the 
generalized Hasimoto transformation, and the maps
\( \mbox{\mathversion{bold}$v$}_{0}\mapsto q_{0}\) and 
\( q \mapsto \mbox{\mathversion{bold}$v$} \) implied in 
Theorem \ref{hasi2} as the 
inverse generalized Hasimoto transformation.

%
%


\section{Proof of Main Theorems}
\setcounter{equation}{0}

In this section, we prove Theorem \ref{hasi1}, \ref{hasi2}, and \ref{stability}.
The transformation utilized in the proof of theorem \ref{hasi1} and 
\ref{hasi2} is mostly due to Koiso \cite{12} with some modifications
to accommodate the presence of boundary conditions.

Although it is implied in Theorem \ref{hasi1} and \ref{hasi2}, 
we explicitly state the solvability of problems (\ref{slant}) and 
(\ref{plane}) in the form that we will take for granted throughout 
this paper.
\begin{Th}
Let \( \mbox{\mathversion{bold}$v$}_{0}\in H^{3}(I_{L})\) and
\( \mbox{\mathversion{bold}$b$}\in \mathbf{S}^{2}\) satisfy
\( \lvert \mbox{\mathversion{bold}$v$}_{0}\rvert \equiv 1\) and 
the compatibility condition for {\rm (\ref{slant})} up to order \( 1\).
Then, there exists a unique solution 
\( \mbox{\mathversion{bold}$v$}\in C\big( [0,\infty);H^{3}(I_{L})\big) \cap C^{1}\big( [0,\infty);H^{1}(I_{L})\big) \) of problem {\rm (\ref{slant})}.
Furthermore, \( \lvert \mbox{\mathversion{bold}$v$}(s,t)\rvert =1\) for 
all \( s\in I_{L}\) and \( t>0\). Additionally, there exists 
\( c_{\ast}>0\) such that
\( \mbox{\mathversion{bold}$v$}\) satisfies
\begin{align}
\sup_{t>0} \| \mbox{\mathversion{bold}$v$}(t)\|_{3}
\leq
c_{\ast},
\label{energy}
\end{align}
where \( c_{\ast}>0 \) depends on \( L\) and 
\( \|\mbox{\mathversion{bold}$v$}_{0}\|_{3}\)
and is non-decreasing with respect to 
\( \| \mbox{\mathversion{bold}$v$}_{0}\|_{3}\)
\label{existv}
\end{Th}
\begin{Th}
Let \( q_{0}\in H^{2}(I_{L})\) satisfy the \(0\)-th order 
compatibility condition for {\rm (\ref{plane})}. Then, 
there exists a unique solution 
\( q\in C\big( [0,\infty);H^{2}(I_{L})\big) \cap C^{1}\big( [0,\infty);L^{2}(I_{L})\big) \) of problem {\rm (\ref{plane})}.

\label{existq}
\end{Th}
Theorem \ref{existv} is proved by the author in \cite{23} and 
Theorem \ref{existq} is essentially due to 
Bourgain \cite{41,63} since problem (\ref{plane}) can be reduced to the 
initial value problem on the torus by reflection and 
periodic extension.

\subsection{Proof of Theorem \ref{hasi1}}
\label{th1}

Let \( q_{0}\in H^{2}(I_{L}) \) satisfy the \(0\)-th order
compatibility condition for problem (\ref{plane}) and 
define \( q^{0}_{1}\) and \( q^{0}_{2}\) by
\begin{align*}
q_{0}(s) = q^{0}_{1}(s) + {\rm i}q^{0}_{2}(s).
\end{align*}
Furthermore, we define \( \mbox{\mathversion{bold}$v$}_{0},\mbox{\mathversion{bold}$e$}^{0}\), and \( \mbox{\mathversion{bold}$w$}^{0}\)
as the solution of the following system of ordinary differential equations.
\begin{align*}
\left\{
\begin{array}{ll}
\mbox{\mathversion{bold}$v$}_{0s}=
q^{0}_{1}\mbox{\mathversion{bold}$e$}^{0}
+q^{0}_{2}\mbox{\mathversion{bold}$w$}^{0}, & s\in I_{L}, \\[3mm]
\mbox{\mathversion{bold}$e$}^{0}_{s}=-q^{0}_{1}\mbox{\mathversion{bold}$v$}_{0}, &
s\in I_{L}, \\[3mm]
\mbox{\mathversion{bold}$w$}^{0}_{s}=-q^{2}_{0}\mbox{\mathversion{bold}$v$}_{0}, &
s\in I_{L}, \\[3mm]
(\mbox{\mathversion{bold}$v$}_{0}, \mbox{\mathversion{bold}$e$}^{0},
\mbox{\mathversion{bold}$w$}^{0})(0)=
(\mbox{\mathversion{bold}$e$}_{1},-\mbox{\mathversion{bold}$e$}_{2},
\mbox{\mathversion{bold}$e$}_{3}). & \ 
\end{array}\right.
\end{align*}
By direct calculations, we see that the triplet
\( \{ \mbox{\mathversion{bold}$v$}_{0}, \mbox{\mathversion{bold}$e$}^{0},\mbox{\mathversion{bold}$w$}^{0} \}\) is an orthonormal basis of 
\( \mathbf{R}^{3}\) for all \( s\in I_{L}\). 
Set \( \mbox{\mathversion{bold}$b$}:= \mbox{\mathversion{bold}$v$}_{0}(L)\),
and we see that by definition, 
\( \mbox{\mathversion{bold}$v$}_{0}\) and \( \mbox{\mathversion{bold}$b$} \)
satisfy the \( 0\)-th order
compatibility condition of problem (\ref{slant}).
We further calculate and see that
\begin{align*}
\mbox{\mathversion{bold}$v$}_{0ss}=
(-(q^{0}_{1})^{2}-(q^{0}_{2})^{2})\mbox{\mathversion{bold}$v$}_{0}
+ q^{0}_{1s}\mbox{\mathversion{bold}$e$}^{0}
+ q^{0}_{2s}\mbox{\mathversion{bold}$w$}^{0},
\end{align*}
and hence,
\begin{align*}
\mbox{\mathversion{bold}$v$}_{0}\times \mbox{\mathversion{bold}$v$}_{0ss}
&=
q^{0}_{1s}\mbox{\mathversion{bold}$v$}_{0}\times \mbox{\mathversion{bold}$e$}^{0}
+ q^{0}_{2s}\mbox{\mathversion{bold}$v$}_{0}\times \mbox{\mathversion{bold}$w$}^{0}
\\[3mm]
&=
-q^{0}_{1s}\mbox{\mathversion{bold}$w$}^{0}
+q^{0}_{2s}\mbox{\mathversion{bold}$e$}^{0}.
\end{align*}
Here, we have used the fact that 
\( \{ \mbox{\mathversion{bold}$v$}_{0}, \mbox{\mathversion{bold}$e$}^{0},\mbox{\mathversion{bold}$w$}^{0} \} \) is an orthonormal basis 
of \( \mathbf{R}^{3}\) and the orientation given at \( s=0 \)
yields 
\( \mbox{\mathversion{bold}$v$}_{0}\times \mbox{\mathversion{bold}$e$}^{0} = -\mbox{\mathversion{bold}$w$}^{0}\)
and
\( \mbox{\mathversion{bold}$v$}_{0}\times \mbox{\mathversion{bold}$w$}^{0}=\mbox{\mathversion{bold}$e$}^{0}\).
Since \( q_{0}\) satisfies the \(0\)-th order compatibility condition
for (\ref{plane}), 
we see that 
\begin{align*}
\mbox{\mathversion{bold}$v$}_{0}(0)\times \mbox{\mathversion{bold}$v$}_{0ss}(0)
=
\mbox{\mathversion{bold}$v$}_{0}(L)\times \mbox{\mathversion{bold}$v$}_{0ss}(L)
=
\mbox{\mathversion{bold}$0$}
\end{align*}
and \( \mbox{\mathversion{bold}$v$}_{0}\) and
\( \mbox{\mathversion{bold}$b$}\) satisfies the \( 1\)-st order
compatibility condition for (\ref{slant}).

Let, \( \mbox{\mathversion{bold}$v$}\in C\big( [0,\infty);H^{3}(I_{L})\big) \cap C^{1}\big( [0,\infty);H^{1}(I_{L})\big) \) be the solution of 
problem (\ref{slant}) with initial datum \( \mbox{\mathversion{bold}$v$}_{0}\)
and boundary datum \( \mbox{\mathversion{bold}$b$}\) just obtained.
From Theorem \ref{existv}, we know that 
\( \lvert \mbox{\mathversion{bold}$v$}(s,t) \rvert =1\) 
for all \( s\in I_{L}\) and \( t>0\).
We now define \( \mbox{\mathversion{bold}$e$} \) as the solution of 
\begin{align*}
\left\{
\begin{array}{ll}
\mbox{\mathversion{bold}$e$}_{s}=-(\mbox{\mathversion{bold}$v$}_{s}\cdot
\mbox{\mathversion{bold}$e$})\mbox{\mathversion{bold}$v$}, & s\in I_{L}, \\[3mm]
\mbox{\mathversion{bold}$e$}(0) = \mbox{\mathversion{bold}$e$}_{2}, & \
\end{array}\right.
\end{align*}
and set \( \mbox{\mathversion{bold}$w$}=\mbox{\mathversion{bold}$v$}\times\mbox{\mathversion{bold}$e$}\). 
Here, \( \cdot \) is the standard inner product of \( \mathbf{R}^{3}\).
Note that 
\(\mbox{\mathversion{bold}$e$}\) and 
\( \mbox{\mathversion{bold}$w$}\) depend on \( t>0\) as a parameter
through \( \mbox{\mathversion{bold}$v$}\).
Again, we see that 
\( \{ \mbox{\mathversion{bold}$v$},\mbox{\mathversion{bold}$e$},\mbox{\mathversion{bold}$w$}\} \) is an orthonormal basis of
\( \mathbf{R}^{3}\)
for all \( s\in I_{L}\) and \( t>0 \).
Since \( \mbox{\mathversion{bold}$v$}\cdot \mbox{\mathversion{bold}$v$}_{s}\equiv 0\), 
we have the decomposition
\begin{align}
\mbox{\mathversion{bold}$v$}_{s}=
\psi _{1}\mbox{\mathversion{bold}$e$}+\psi_{2}\mbox{\mathversion{bold}$w$}
\label{vs}
\end{align}
for some \( \psi_{1}(s,t) \) and \( \psi_{2}(s,t)\).
Then we also see that
\begin{align}
\mbox{\mathversion{bold}$e$}_{s} &= -\psi_{1}\mbox{\mathversion{bold}$v$}, 
\label{es}\\[3mm]
\mbox{\mathversion{bold}$w$}_{s} &= -\psi_{2}\mbox{\mathversion{bold}$v$}.
\label{ws}
\end{align}
Taking the inner product of 
\( \mbox{\mathversion{bold}$e$} \) and 
\( \mbox{\mathversion{bold}$w$} \) with equation
(\ref{vs}) we have
\begin{align*}
\psi_{1}&=\mbox{\mathversion{bold}$e$}
\cdot \mbox{\mathversion{bold}$v$}_{s}, \\[3mm]
\psi_{2}&=\mbox{\mathversion{bold}$w$}
\cdot \mbox{\mathversion{bold}$v$}_{s},
\end{align*}
from which we deduce that
\( \psi_{1},\psi_{2}\in C\big( [0,\infty);H^{2}(I_{L})\big) \) since
\( \mbox{\mathversion{bold}$v$}\in C\big( [0,\infty);H^{3}(I_{L})\big)  \).
From equation (\ref{vs}), we see that
\begin{align*}
\lvert \psi_{1}\rvert ^{2} + \lvert \psi_{2}\rvert ^{2}
=
\lvert \mbox{\mathversion{bold}$v$}_{s}\rvert ^{2},
\end{align*}
which implies
\begin{align*}
\|\psi_{1}(t)\|^{2} + \|\psi_{2}(t)\|^{2} 
=
\| \mbox{\mathversion{bold}$v$}_{s}(t)\|^{2}.
\end{align*}
Furthermore, we have
\begin{align*}
\psi_{1s}\mbox{\mathversion{bold}$e$} + \psi_{2s}\mbox{\mathversion{bold}$w$}
&=
\mbox{\mathversion{bold}$v$}_{ss} + (\psi_{1})^{2}\mbox{\mathversion{bold}$v$}
+
(\psi_{2})^{2}\mbox{\mathversion{bold}$v$}, \\[3mm]
\psi_{1ss}\mbox{\mathversion{bold}$e$}+\psi_{2ss}\mbox{\mathversion{bold}$w$}
&=
\mbox{\mathversion{bold}$v$}_{sss}
+
3\psi_{1}\psi_{1s}\mbox{\mathversion{bold}$v$}
+
3\psi_{2}\psi_{2s}\mbox{\mathversion{bold}$v$}
+
\psi_{1}\big( (\psi_{1})^{2}+(\psi_{2})^{2}\big)\mbox{\mathversion{bold}$e$}\\[3mm]
& \hspace{10mm} +
\psi_{2} \big( (\psi_{1})^{2}+(\psi_{2})^{2}\big)\mbox{\mathversion{bold}$w$},
\end{align*}
and after taking the inner product of the above equations with 
\( \mbox{\mathversion{bold}$e$}\) and \( \mbox{\mathversion{bold}$w$} \),
we derive
\begin{align*}
\| \psi_{1s}(t)\| &\leq \| \mbox{\mathversion{bold}$v$}_{ss}(t)\|, \\[3mm]
\| \psi_{2s}(t)\| &\leq \| \mbox{\mathversion{bold}$v$}_{ss}(t)\|, \\[3mm]
\| \psi_{1ss}(t)\| &\leq \| \mbox{\mathversion{bold}$v$}_{sss}(t)\| 
+ C\| \mbox{\mathversion{bold}$v$}_{s}(t)\|_{1}^{2}
\| \mbox{\mathversion{bold}$v$}_{s}(t)\|, \\[3mm]
\| \psi_{2ss}(t)\| &\leq \| \mbox{\mathversion{bold}$v$}_{sss}(t)\| 
+ C\| \mbox{\mathversion{bold}$v$}_{s}(t)\|_{1}^{2}
\| \mbox{\mathversion{bold}$v$}_{s}(t)\|,
\end{align*}
where \( C>0\) is determined from the embedding
\(  H^{1}(I_{L}) \hookrightarrow L^{\infty}(I_{L}) \).
Hence, we have
\begin{align}
\| \psi_{1}(t)\|_{2} + \| \psi_{2}(t)\|_{2}
\leq
C\big(
\| \mbox{\mathversion{bold}$v$}(t) \|_{3} 
+
\| \mbox{\mathversion{bold}$v$}(t) \|_{3}^{3}
\big),
\label{psiest}
\end{align}
for all \( t>0\), where \( C>0\) is independent of
\( t\) and \( \mbox{\mathversion{bold}$v$}\).

Since \( \mbox{\mathversion{bold}$v$}\cdot \mbox{\mathversion{bold}$v$}_{t}\equiv 0\), a decomposition of the form
\( \mbox{\mathversion{bold}$v$}_{t}=p_{1}\mbox{\mathversion{bold}$e$} + p_{2}\mbox{\mathversion{bold}$w$}\)
holds. From the equation \( (\mbox{\mathversion{bold}$v$}_{t})_{s} = (\mbox{\mathversion{bold}$v$}_{s})_{t} \) and
\( (\mbox{\mathversion{bold}$e$}_{t})_{s} = (\mbox{\mathversion{bold}$e$}_{s})_{t}\),
 we deduce that \( p_{1} = -\psi_{2s}\) and \( p_{2}=\psi_{1s}\).
Furthermore, 
\begin{align}
\mbox{\mathversion{bold}$v$}_{t}&=-\psi_{2s}\mbox{\mathversion{bold}$e$}
+ \psi_{1s}\mbox{\mathversion{bold}$w$}, \label{vt} \\[3mm]
\mbox{\mathversion{bold}$e$}_{t}&=\psi_{2s}\mbox{\mathversion{bold}$v$}
+
\bigg\{
-\frac{1}{2}\big((\psi_{1})^{2}+(\psi_{2})^{2}\big)
+\frac{1}{2}\big((\psi_{1})^{2}+(\psi_{2})^{2}\big)\big\rvert_{s=0} \bigg\}
\mbox{\mathversion{bold}$w$}, \nonumber \\[3mm]
\mbox{\mathversion{bold}$w$}_{t} 
&=
-\psi_{1s}\mbox{\mathversion{bold}$v$}
-
\bigg\{
-\frac{1}{2}\big((\psi_{1})^{2}+(\psi_{2})^{2}\big)
+\frac{1}{2}\big((\psi_{1})^{2}+(\psi_{2})^{2}\big)\big\rvert_{s=0} \bigg\}
\mbox{\mathversion{bold}$e$}, \nonumber 
\end{align}
holds. Here, \( \big\rvert_{s=0} \) denotes the trace at \( s=0\).
The above three equations along with the fact that
\( \mbox{\mathversion{bold}$v$}\in C^{1}\big([0,\infty);H^{1}(I_{L})\big) \), 
\( \psi_{1}=\mbox{\mathversion{bold}$e$}\cdot \mbox{\mathversion{bold}$v$}_{s}\),
and \( \psi_{2}=\mbox{\mathversion{bold}$w$}\cdot \mbox{\mathversion{bold}$v$}_{s}\)
implies that
\( \psi_{1},\psi_{2}\in C^{1}\big([0,\infty);L^{2}(I_{L}) \big) \).
This in turn allows us to calculate as follows.
\begin{align*}
\psi_{1t}&=-\psi_{2ss} +\bigg\{
-\frac{1}{2}\big((\psi_{1})^{2}+(\psi_{2})^{2}\big)
+\frac{1}{2}\big((\psi_{1})^{2}+(\psi_{2})^{2}\big)\big\rvert_{s=0} 
\bigg\}\psi_{2}
, \\[3mm]
\psi_{2t}
&= \psi_{1ss}
- \bigg\{
-\frac{1}{2}\big((\psi_{1})^{2}+(\psi_{2})^{2}\big)
+\frac{1}{2}\big((\psi_{1})^{2}+(\psi_{2})^{2}\big) \big\rvert_{s=0} 
\bigg\}\psi_{1}.
\end{align*}
Setting \( \psi = \psi_{1}+{\rm i}\psi_{2} \), we see that
\begin{align*}
{\rm i}\psi_{t}-\psi_{ss} = \frac{1}{2}\lvert \psi \rvert^{2}\psi 
- \frac{1}{2}\lvert \psi(0,t)\rvert ^{2}\psi,
\end{align*}
and after a gauge transform given by
\begin{align}
q(s,t)=\psi(s,t) \exp \bigg\{ \frac{{\rm i}}{2}\int^{t}_{0}
\lvert \psi(0,\tau )\rvert ^{2}
 \ d\tau\bigg\},
\label{gauge}
\end{align}
we see that \( q\) satisfies
\begin{align*}
{\rm i}q_{t} = q_{ss}+\frac{1}{2}\lvert q\rvert^{2}q.
\end{align*}
As \( t\) tends to zero in equation (\ref{vs}), we see 
from the uniqueness of the solution of ordinary differential equations that
\begin{align*}
q(s,0)=q_{0}(s).
\end{align*}
Finally, we see that by taking the derivative of the boundary condition 
in problem (\ref{slant}) with respect to \( t\), we see that
\( \mbox{\mathversion{bold}$v$}_{t}(0,t) = \mbox{\mathversion{bold}$v$}_{t}(L,t) =\mbox{\mathversion{bold}$0$} \), and from equation (\ref{vt}), we see that
\begin{align*}
q_{s}(0,t) = q_{s}(L,t) = 0.
\end{align*}
Also note that \( q\in C\big( [0,\infty);H^{2}(I_{L})\big) \cap C^{1}\big( [0,\infty);L^{2}(I_{L})\big) \) from the definition 
of \( q\) along with the fact that
\( \psi_{1},\psi_{2}\in C\big( [0,\infty);H^{2}(I_{L})\big) \cap C^{1}\big( [0,\infty);L^{2}(I_{L})\big) \).
Additionally, from equations (\ref{psiest}) and (\ref{gauge}) we have
\begin{align}
\| q(t)\|_{2}
\leq
C\big( \| \mbox{\mathversion{bold}$v$}(t)\|_{3}
+
\| \mbox{\mathversion{bold}$v$}(t)\|^{3}_{3} \big),
\label{qvestimate}
\end{align}
for all \( t\geq 0 \), where \( C>0\) is independent of \( t\) and 
\( \mbox{\mathversion{bold}$v$}\). Estimate (\ref{qvestimate}) will be
utilized later.
This finishes the proof of Theorem \ref{hasi1}. \\
\hfill \(  \Box\)


\subsection{Proof of Theorem \ref{hasi2} }

Let \( \mbox{\mathversion{bold}$v$}_{0}\in H^{3}(I_{L}) \) satisfying
\( \lvert \mbox{\mathversion{bold}$v$}_{0}\rvert \equiv 1\) 
and \( \mbox{\mathversion{bold}$b$} \in \mathbf{S}^{2}\) be 
initial and boundary datum for problem (\ref{slant}) satisfying the 
compatibility conditions up to order \( 1\). 
Define \( \mbox{\mathversion{bold}$e$}^{0} \) as the solution of
\begin{align*}
\left\{
\begin{array}{ll}
\mbox{\mathversion{bold}$e$}^{0}_{s}=-(\mbox{\mathversion{bold}$v$}_{0s}
\cdot \mbox{\mathversion{bold}$e$}^{0})\mbox{\mathversion{bold}$v$}_{0}, 
& s\in I_{L}, \\[3mm]
\mbox{\mathversion{bold}$e$}^{0}(0)=-\mbox{\mathversion{bold}$e$}_{2}, & \ 
\end{array}\right.
\end{align*}
and set \( \mbox{\mathversion{bold}$w$}^{0} = \mbox{\mathversion{bold}$e$}^{0} \times \mbox{\mathversion{bold}$v$}_{0} \).
Similarly to Section \ref{th1}, the triplet
\( \{ \mbox{\mathversion{bold}$v$}_{0}, \mbox{\mathversion{bold}$e$}^{0},\mbox{\mathversion{bold}$w$}^{0}\} \) is an orthonormal basis of 
\( \mathbf{R}^{3}\) and the decomposition 
\begin{align*}
\mbox{\mathversion{bold}$v$}_{0s} = 
q^{0}_{1}\mbox{\mathversion{bold}$e$}^{0} + 
q^{0}_{2}\mbox{\mathversion{bold}$w$}^{0}
\end{align*}
holds and direct calculation yields
\begin{align*}
\mbox{\mathversion{bold}$v$}_{0}\times \mbox{\mathversion{bold}$v$}_{0ss}
=
q^{0}_{2s}\mbox{\mathversion{bold}$e$}^{0} 
- q^{0}_{1s} \mbox{\mathversion{bold}$w$}^{0}.
\end{align*}
Hence, setting \( q_{0}=q^{0}_{1} + {\rm i}q^{0}_{2} \)
we see that \( q_{0}\) satisfies \( q_{0s}(0,t)=q_{0s}(L,t)=0\) 
since \( \mbox{\mathversion{bold}$v$}_{0} \) and 
\( \mbox{\mathversion{bold}$b$}\) satisfy the \( 1\)-st 
order compatibility condition for (\ref{slant}).
In other words, \( q_{0}\) satisfies the \( 0\)-th order 
compatibility condition for (\ref{plane}).

Let \( q\in C\big( [0,\infty);H^{2}(I_{L})\big) \cap C^{1}\big( [0,\infty);L^{2}(I_{L})\big) \) be the solution of 
problem (\ref{plane}) with initial datum \( q_{0}\) just obtained
and set \( q(s,t)=q_{1}(s,t) + {\rm i}q_{2}(s,t) \).
We define \( \mbox{\mathversion{bold}$v$}\), \( \mbox{\mathversion{bold}$e$} \),
and \( \mbox{\mathversion{bold}$w$} \) as the solution of the following 
system of ordinary differential equations with respect to \( s\).
\begin{align}
\left\{
\begin{array}{ll}
\mbox{\mathversion{bold}$v$}_{s}=q_{1}\mbox{\mathversion{bold}$e$}
+ q_{2}\mbox{\mathversion{bold}$w$}, & s\in I_{L}, \\[3mm]
\mbox{\mathversion{bold}$e$}_{s}=-q_{1}\mbox{\mathversion{bold}$v$}, & 
s\in I_{L}, \\[3mm]
\mbox{\mathversion{bold}$w$}_{s} = -q_{2}\mbox{\mathversion{bold}$v$}, &
s\in I_{L}, \\[3mm]
(\mbox{\mathversion{bold}$v$},\mbox{\mathversion{bold}$e$},
\mbox{\mathversion{bold}$w$})(0)=
(\mbox{\mathversion{bold}$e$}_{1},-\mbox{\mathversion{bold}$e$}_{2},
\mbox{\mathversion{bold}$e$}_{3}). & \ 
\end{array}\right.
\label{vss}
\end{align}
Again, \( \{ \mbox{\mathversion{bold}$v$},\mbox{\mathversion{bold}$e$},\mbox{\mathversion{bold}$w$} \} \) depends on \( t>0\) as a parameter
and \( \{ \mbox{\mathversion{bold}$v$}, \mbox{\mathversion{bold}$e$},\mbox{\mathversion{bold}$w$}\} \) is an orthonormal basis of 
\( \mathbf{R}^{3}\) for all \( s\in I_{L}\) and \( t>0\).
From (\ref{vss}), we deduce that
\( \mbox{\mathversion{bold}$v$}\in C\big([0,\infty);H^{3}(I_{L})\big)\cap C^{1}\big([0,\infty);H^{1}(I_{L})\big) \) and
\begin{align*}
\| \mbox{\mathversion{bold}$v$}_{s}(t)\|^{2} 
&= \| q_{1}(t) \|^{2} + \|q_{2}(t)\|^{2}, \\[3mm]
\| \mbox{\mathversion{bold}$v$}_{ss}(t)\|
&\leq
C\big( 
\|q_{1s}(t)\|+ \| q_{2s}(t)\| 
+ \|q_{1}(t)\|^{2}_{1}+\|q_{2}(t)\|^{2}_{1}\big), \\[3mm]
\| \mbox{\mathversion{bold}$v$}_{sss}(t)\|
&\leq
C\big( 
\|q_{1ss}(t)\| + \| q_{2ss}(t)\| 
+ \|q_{1}(t)\|^{3}_{2}+\|q_{2}(t)\|^{3}_{2}\big),
\end{align*}
where \( C>0\) is independent of \( t\) and \( q\).
Hence we have
\begin{align*}
\lvert \mbox{\mathversion{bold}$v$}(s,t)\rvert &=1 \ \text{for all} \ s\in I_{L} \
\text{and} \ t\geq 0, \\[3mm]
\|\mbox{\mathversion{bold}$v$}_{s}(t)\|_{2}
&\leq
C\big(
\|q(t)\|_{2} + \| q(t)\|^{3}_{2}
\big).
\end{align*}
Note that the definition of \( \mbox{\mathversion{bold}$v$}_{0}\),
\( \mbox{\mathversion{bold}$e$}^{0}\), and \(\mbox{\mathversion{bold}$w$}^{0}\)
coincides with the limit as \( t\) tends to zero in equations (\ref{vss}).

Next, we see that
\begin{align}
\mbox{\mathversion{bold}$v$}_{t}&=q_{2s}\mbox{\mathversion{bold}$e$}
-q_{1s}\mbox{\mathversion{bold}$w$}, \label{vtt} \\[3mm]
\mbox{\mathversion{bold}$e$}_{t}&=-q_{2s}\mbox{\mathversion{bold}$v$}
+(\frac{1}{2}\lvert q\rvert^{2}-\frac{1}{2}\lvert q(0,t)\rvert^{2})\mbox{\mathversion{bold}$w$},  
\nonumber \\[3mm]
\mbox{\mathversion{bold}$w$}_{t}&=q_{1s}\mbox{\mathversion{bold}$v$}
-(\frac{1}{2}\lvert q\rvert^{2}-\frac{1}{2}\lvert q(0,t)\rvert^{2})\mbox{\mathversion{bold}$e$}.
\nonumber
\end{align}
From equations (\ref{vss}), we see that
\begin{align*}
\mbox{\mathversion{bold}$v$}\times \mbox{\mathversion{bold}$v$}_{ss}
=
q_{2s}\mbox{\mathversion{bold}$e$}-q_{1s}\mbox{\mathversion{bold}$w$},
\end{align*}
and hence, along with (\ref{vtt}) we have
\begin{align*}
\mbox{\mathversion{bold}$v$}_{t}=\mbox{\mathversion{bold}$v$}\times 
\mbox{\mathversion{bold}$v$}_{ss}.
\end{align*}
The first equation in (\ref{vss}) with \( t=0\) together with 
the uniqueness of the solution of ordinary differential equations
show that
\begin{align*}
\mbox{\mathversion{bold}$v$}(s,0) = \mbox{\mathversion{bold}$v$}_{0}(s).
\end{align*}
Finally, taking the trace at \( s=0\) and \( s=L\) in equation (\ref{vtt})
we see that
\( \mbox{\mathversion{bold}$v$}_{t}(0,t)=\mbox{\mathversion{bold}$v$}_{t}(L,t)= \mbox{\mathversion{bold}$0$}\) which in turn yields
\begin{align*}
\mbox{\mathversion{bold}$v$}(0,t) &= \mbox{\mathversion{bold}$v$}_{0}(0)=\mbox{\mathversion{bold}$e$}_{1}, \\[3mm]
\mbox{\mathversion{bold}$v$}(L,t) &= \mbox{\mathversion{bold}$v$}_{0}(L)=\mbox{\mathversion{bold}$b$},
\end{align*}
and hence, \( \mbox{\mathversion{bold}$v$}\) satisfies the 
boundary condition of problem (\ref{slant}). 
This shows that \( \mbox{\mathversion{bold}$v$}\) is the desired solution of
problem (\ref{slant}). 
This finishes the proof of Theorem \ref{hasi2}. \hfill \( \Box\)


\subsection{Proof of Theorem \ref{stability}}
We divide the proof of Theorem \ref{stability} into five steps.

We first establish that perturbations to the initial datum of 
problem (\ref{plane}) correspond to perturbations to the initial datum
of problem (\ref{slant}). Furthermore, we derive estimates that
show in what norms the perturbations translate.

Secondly, we show that \( q_{R}(0) \) corresponds to a particular type of 
initial datum of problem (\ref{slant}). Solutions of problem (\ref{slant})
with this type of initial datum will be referred to as an
arc-shaped solution of problem (\ref{slant}).

Then, we prove stability estimates for 
arc-shaped solutions of problem (\ref{slant}),
which is related to the stability problem for plane wave solutions of
problem (\ref{plane}). The stability estimates for arc-shaped solutions
are already utilized by the author in \cite{27}, but we reiterate it here for completeness.

Finally, we combine the results of the previous steps to prove that 
the plane wave solution \( q_{R}\) is orbitally stable. More precisely, 
we show that the stability estimates for problem (\ref{slant}) 
yield stability estimates for problem (\ref{plane}) through the 
generalized Hasimoto transformation. 
%


\subsubsection{The Generalized Hasimoto Transformation of Perturbations}
\label{inverse hasimoto}

In this section, we show that perturbations to the initial datum 
of problem (\ref{plane}) translate to perturbations
to the initial datum of problem (\ref{slant}) through the
generalized Hasimoto transformation.

Let \( q_{0}\in H^{2}(I_L) \) and \( \varphi_{0}\in H^{2}(I_L) \) be 
such that \( q_{0} \) and \( \varphi_{0}\) satisfy the 
\( 0 \)-th order compatibility condition for problem (\ref{plane}).
Note that \( q_{0}+\varphi_{0}\) also satisfy the same compatibility condition.
Setting \( \tilde{q}_{0} = q_{0}+\varphi_{0}\) and 
\begin{align*}
q_{0}=q^{0}_{1}+{\rm i}q^{0}_{2} \qquad \text{and} \qquad
\tilde{q}_{0}=\tilde{q}^{0}_{1}+{\rm i}\tilde{q}^{0}_{2},
\end{align*}
define \( \{\mbox{\mathversion{bold}$v$}_{0},\mbox{\mathversion{bold}$e$}^{0},\mbox{\mathversion{bold}$w$}^{0} \}\) and
\( \{ \tilde{\mbox{\mathversion{bold}$v$}}_{0},\tilde{\mbox{\mathversion{bold}$e$}}^{0},\tilde{\mbox{\mathversion{bold}$w$}}^{0} \} \)
as the solution of
\begin{align*}
\left\{
\begin{array}{ll}
\mbox{\mathversion{bold}$v$}_{0s} = q^{0}_{1}\mbox{\mathversion{bold}$e$}^{0}
+ q^{0}_{2}\mbox{\mathversion{bold}$w$}^{0}, & s\in I_{L}, \\[3mm]
\mbox{\mathversion{bold}$e$}^{0}_{s} 
= -q^{0}_{1}\mbox{\mathversion{bold}$v$}_{0}, & s\in I_{L}, \\[3mm]
\mbox{\mathversion{bold}$w$}^{0}_{s}
= -q^{0}_{2}\mbox{\mathversion{bold}$v$}_{0}, & s\in I_{L}, \\[3mm]
(\mbox{\mathversion{bold}$v$}_{0},\mbox{\mathversion{bold}$e$}^{0},
\mbox{\mathversion{bold}$w$}^{0})
=
(\mbox{\mathversion{bold}$e$}_{1},
-\mbox{\mathversion{bold}$e$}_{2},
\mbox{\mathversion{bold}$e$}_{3}), & \ 
\end{array}\right.
\\[5mm]
\left\{
\begin{array}{ll}
\tilde{\mbox{\mathversion{bold}$v$}}_{0s} 
= q^{0}_{1}\tilde{\mbox{\mathversion{bold}$e$}}^{0}
+ q^{0}_{2}\tilde{\mbox{\mathversion{bold}$w$}}^{0}, & s\in I_{L}, \\[3mm]
\tilde{\mbox{\mathversion{bold}$e$}}^{0}_{s} 
= -q^{0}_{1}\tilde{\mbox{\mathversion{bold}$v$}}_{0}, & s\in I_{L}, \\[3mm]
\tilde{\mbox{\mathversion{bold}$w$}}^{0}_{s}
= -q^{0}_{2}\tilde{\mbox{\mathversion{bold}$v$}}_{0}, & s\in I_{L}, \\[3mm]
(\tilde{\mbox{\mathversion{bold}$v$}}_{0},
\tilde{\mbox{\mathversion{bold}$e$}}^{0},
\tilde{\mbox{\mathversion{bold}$w$}}^{0})
=
(\mbox{\mathversion{bold}$e$}_{1},
-\mbox{\mathversion{bold}$e$}_{2},
\mbox{\mathversion{bold}$e$}_{3}), & \ 
\end{array}\right.
\end{align*}
respectively. Setting
\( \mbox{\mathversion{bold}$V$}:= \mbox{\mathversion{bold}$v$}_{0} - \tilde{\mbox{\mathversion{bold}$v$}}_{0}, \ \mbox{\mathversion{bold}$E$}:= \mbox{\mathversion{bold}$e$}^{0}-\tilde{\mbox{\mathversion{bold}$e$}}^{0},\) and \(\mbox{\mathversion{bold}$W$}:=\mbox{\mathversion{bold}$w$}^{0}-\tilde{\mbox{\mathversion{bold}$w$}}^{0} \),
we see that
\begin{align}
\left\{
\begin{array}{ll}
\mbox{\mathversion{bold}$V$}_{s}
=
(q^{0}_{1}-\tilde{q}^{0}_{1})\mbox{\mathversion{bold}$e$}^{0}
+
\tilde{q}^{0}_{1}\mbox{\mathversion{bold}$E$}
+
(q^{0}_{2}-\tilde{q}^{0}_{2})\mbox{\mathversion{bold}$w$}^{0}
+
\tilde{q}^{0}_{2}\mbox{\mathversion{bold}$W$}, & s\in I_{L}, \\[3mm]
\mbox{\mathversion{bold}$E$}_{s}
=
(q^{0}_{1}-\tilde{q}^{0}_{1})\mbox{\mathversion{bold}$v$}_{0}
-
\tilde{q}^{0}_{1}\mbox{\mathversion{bold}$V$}, & s\in I_{L}, \\[3mm]
\mbox{\mathversion{bold}$W$}_{s}
=
(\tilde{q}^{0}_{2}-q^{0}_{2})\mbox{\mathversion{bold}$v$}_{0}
-
\tilde{q}^{0}_{2}\mbox{\mathversion{bold}$V$}, & s\in I_{L}, \\[3mm]
(\mbox{\mathversion{bold}$V$},\mbox{\mathversion{bold}$E$},
\mbox{\mathversion{bold}$W$})(0)
=
(\mbox{\mathversion{bold}$0$},\mbox{\mathversion{bold}$0$},
\mbox{\mathversion{bold}$0$}),
\end{array}\right.
\label{VEW}
\end{align}
holds. Standard energy estimates yield
\begin{align*}
\frac{1}{2}\frac{{\rm d}}{{\rm d}s}
\bigg\{
\lvert\mbox{\mathversion{bold}$V$}\rvert^{2}&+\lvert\mbox{\mathversion{bold}$E$}\rvert^{2}
+ \lvert \mbox{\mathversion{bold}$W$}\rvert ^{2}
\bigg\} \\[3mm]
&\leq
C\big(1+
\| \tilde{q}^{0}_{1}\|^{2}_{L^{\infty}(I_{L})}
+\| \tilde{q}^{0}_{2}\|^{2}_{L^{\infty}(I_{L})}
\big)
\big( \lvert\mbox{\mathversion{bold}$V$}\rvert^{2}+\lvert\mbox{\mathversion{bold}$E$}\rvert^{2}
+ \lvert \mbox{\mathversion{bold}$W$}\rvert ^{2}\big) \\[3mm]
& \qquad +
\big( 
\| \tilde{q}^{0}_{1}-q^{0}_{1}\|^{2}_{L^{\infty}(I_{L})}
+\| \tilde{q}^{0}_{2}-q^{0}_{2}\|^{2}_{L^{\infty}(I_{L})}
\big),
\end{align*}
and from Gronwall's inequality, we have
\begin{align}
\lvert\mbox{\mathversion{bold}$V$}\rvert^{2}
+\lvert\mbox{\mathversion{bold}$E$}\rvert^{2}
+ \lvert\mbox{\mathversion{bold}$W$}\rvert^{2}
\leq
Ce^{CL}
\bigg\{
\| \tilde{q}^{0}_{1}-q^{0}_{1}\|^{2}_{L^{\infty}(I_{L})}
+\| \tilde{q}^{0}_{2}-q^{0}_{2}\|^{2}_{L^{\infty}(I_{L})}
\bigg\},
\label{baseest}
\end{align}
where \( C>0\) depends on \( \| q_{0}\|_{L^{\infty}(I_{L})}\) 
and \( \| \tilde{q}_{0}\|_{L^{\infty}(I_{L})} \).
Hence, estimate (\ref{baseest}) implies that
if \( \tilde{q}_{0}-q_{0}\) is controlled in \( L^{\infty}(I_{L}) \),
then we have a control on
\( \mbox{\mathversion{bold}$V$}, \mbox{\mathversion{bold}$E$},\) and 
\( \mbox{\mathversion{bold}$W$}\)
in \( L^{\infty}(I_{L}) \), and as a consequence, in \( L^{2}(I_{L}) \).
Equations (\ref{VEW}) together with inequality 
(\ref{baseest}) show that
\begin{align*}
&\|\mbox{\mathversion{bold}$V$}_{s}\|^{2}
+\|\mbox{\mathversion{bold}$E$}_{s}\|^{2}
+ \|\mbox{\mathversion{bold}$W$}_{s}\|^{2}
\leq
C\big(
\| \tilde{q}^{0}_{1}-q^{0}_{1}\|^{2}_{L^{\infty}(I_{L})}
+\| \tilde{q}^{0}_{2}-q^{0}_{2}\|^{2}_{L^{\infty}(I_{L})}
\big), \\[3mm]
&\|\mbox{\mathversion{bold}$V$}_{s}\|_{L^{\infty}(I_{L})}
+\|\mbox{\mathversion{bold}$E$}_{s}\|_{L^{\infty}(I_{L})}
+ \|\mbox{\mathversion{bold}$W$}_{s}\|_{L^{\infty}(I_{L})}
\\[3mm] 
& \hspace*{40mm}\leq
C\big(
\| \tilde{q}^{0}_{1}-q^{0}_{1}\|_{L^{\infty}(I_{L})}
+\| \tilde{q}^{0}_{2}-q^{0}_{2}\|_{L^{\infty}(I_{L})}\big),
\end{align*}
hold, where \( C>0\) depends on \( \| q_{0}\|_{L^{\infty}(I_{L})}\) 
and \( \| \tilde{q}_{0}\|_{L^{\infty}(I_{L})} \).

Taking the derivative with respect to \( s\) of equations
(\ref{VEW}), we have
\begin{align*}
\mbox{\mathversion{bold}$V$}_{ss}&=
(q^{0}_{1}-\tilde{q}^{0}_{1})_{s}\mbox{\mathversion{bold}$e$}^{0}
-q^{0}_{1}(q^{0}_{1}-\tilde{q}^{0}_{1})\mbox{\mathversion{bold}$v$}_{0}
+\tilde{q}^{0}_{1s}\mbox{\mathversion{bold}$E$}
+
\tilde{q}^{0}_{1}
\big(
(q^{0}_{1}-\tilde{q}^{0}_{1})\mbox{\mathversion{bold}$v$}_{0}
-
\tilde{q}^{0}_{1}\mbox{\mathversion{bold}$V$}
\big)\\[3mm]
& \qquad +
(q^{0}_{2}-\tilde{q}^{0}_{2})_{s}\mbox{\mathversion{bold}$w$}^{0}
-q^{0}_{2}(q^{0}_{2}-\tilde{q}^{0}_{2})\mbox{\mathversion{bold}$v$}_{0}
+
\tilde{q}^{0}_{2s}\mbox{\mathversion{bold}$W$}
+
\tilde{q}^{0}_{2}\big( (\tilde{q}^{0}_{2}-q^{0}_{2})\mbox{\mathversion{bold}$v$}_{0}
-(\tilde{q}^{0}_{2})\mbox{\mathversion{bold}$V$}
\big) \\[3mm]
&=
-\big(
(q^{0}_{1}-\tilde{q}^{0}_{1})^{2}
+(q^{0}_{2}-\tilde{q}^{0}_{2})^{2}\big) \mbox{\mathversion{bold}$v$}_{0}
+
(q^{0}_{1}-\tilde{q}^{0}_{1})_{s}\mbox{\mathversion{bold}$e$}^{0}
+(q^{0}_{2}-\tilde{q}^{0}_{2})_{s}\mbox{\mathversion{bold}$w$}^{0} \\[3mm]
& \qquad +
\tilde{q}^{0}_{1s}\mbox{\mathversion{bold}$E$}
+ \tilde{q}^{0}_{2s}\mbox{\mathversion{bold}$W$}
-
\big( (\tilde{q}^{0}_{1})^{2}+ (\tilde{q}^{0}_{2})^{2}\big)
\mbox{\mathversion{bold}$V$},
\end{align*}
from which we have the estimate
\begin{align*}
\| \mbox{\mathversion{bold}$V$}_{ss} \|
\leq
C\big(
\| \tilde{q}^{0}_{1}-q^{0}_{1}\|_{1}+ \| \tilde{q}^{0}_{2}-q^{0}_{2}\|_{1}
+
\| \tilde{q}^{0}_{1}-q^{0}_{1}\|^{2}_{1}+ \| \tilde{q}^{0}_{2}-q^{0}_{2}\|^{2}_{1}
\big).
\end{align*}
In a similar fashion, we can derive the estimate
\begin{align*}
\| \mbox{\mathversion{bold}$V$}_{sss} \|
\leq
C\big(
\| \tilde{q}^{0}_{1}-q^{0}_{1}\|_{2}+ \| \tilde{q}^{0}_{2}-q^{0}_{2}\|_{2}
+
\| \tilde{q}^{0}_{1}-q^{0}_{1}\|^{3}_{2}+ \| \tilde{q}^{0}_{2}-q^{0}_{2}\|^{3}_{2}
\big).
\end{align*}
In summary, we have the estimate
\begin{align}
\| \tilde{\mbox{\mathversion{bold}$v$}_{0}}-\mbox{\mathversion{bold}$v$}_{0}\|_{3}
\leq
C\big(
\| \tilde{q}_{0}-q_{0}\|_{2}+ \| \tilde{q}_{0}-q_{0}\|^{3}_{2} 
\big),
\label{pertest}
\end{align}
where \( C>0\) depends on \( \| \tilde{q}_{0}\|_{2}+\|q_{0}\|_{2}\).
We also see that \( C>0\) is non-decreasing with respect to
\( \| \tilde{q}_{0}\|_{2}+\|q_{0}\|_{2}\).
Inequality (\ref{pertest}) shows that \( H^{2}\)-perturbations of \( q_{0} \)
correspond to \( H^{3}\)-perturbations of \( \mbox{\mathversion{bold}$v$}_{0} \).

We summarize the conclusion of this section 
in the following proposition.
\begin{pr}
For \( q_{0}\in H^{2}(I_L) \) and \( \varphi_{0}\in H^{2}(I_L) \) which
satisfy the \( 0 \)-th order compatibility condition for {\rm (\ref{plane})},
let \( \mbox{\mathversion{bold}$v$}_{0}\) and \( \tilde{\mbox{\mathversion{bold}$v$}}_{0}\)
be initial data constructed by the generalized Hasimoto transformation
from \( q_{0}\) and \( q_{0}+\varphi_{0}\), respectively.
Then, \(\mbox{\mathversion{bold}$v$}_{0},\tilde{\mbox{\mathversion{bold}$v$}}_{0}\in H^{3}(I_{L})\) and satsify
\begin{align*}
\| \tilde{\mbox{\mathversion{bold}$v$}_{0}}-\mbox{\mathversion{bold}$v$}_{0}\|_{3}
\leq
C\big(
\| \tilde{q}_{0}-q_{0}\|_{2}+ \| \tilde{q}_{0}-q_{0}\|^{3}_{2} 
\big),
\end{align*}
where \( C>0\) depends on \( \|q_{0}\|_{2}+\|\varphi_{0}\|_{2}\),
and is non-decreasing with respect to \( \|q_{0}\|_{2}+\|\varphi_{0}\|_{2}\).

\end{pr}


\subsubsection{The Inverse Generalized
Hasimoto Transformation of the Plane Wave Solution}
\label{hasiplane}
We investigate the image of the plane wave solution by the inverse
generalized Hasimoto transformation. 
Recall that for \( R>0\), 
\begin{align*}
q_{R}(t)=-\frac{1}{R}\exp\left\{ -\frac{{\rm i}t}{2R^{2}} \right\},
\end{align*}
and \( q_{R}(0) = -\frac{1}{R} \). Following the arguments 
in Section \ref{th1}, the generalized Hasimoto transformation
applied to the initial datum \( -\frac{1}{R} \) corresponds to
defining \( \mbox{\mathversion{bold}$v$}_{0}\) and
\( \mbox{\mathversion{bold}$e$}^{0} \) as the solution
to
\begin{align}
\label{arc}
\left\{
\begin{array}{ll}
\mbox{\mathversion{bold}$v$}_{0s} 
= -\frac{1}{R}\mbox{\mathversion{bold}$e$}^{0}, &  s\in I_{L}, \\[3mm]
\mbox{\mathversion{bold}$e$}^{0}_{s}
= \frac{1}{R}\mbox{\mathversion{bold}$v$}_{0}, & s\in I_{L}, \\[3mm]
(\mbox{\mathversion{bold}$v$}_{0},\mbox{\mathversion{bold}$e$}^{0})(0)
=
(\mbox{\mathversion{bold}$e$}_{1},-\mbox{\mathversion{bold}$e$}_{2}), & \ 
\end{array}\right.
\end{align}
and setting \( \mbox{\mathversion{bold}$w$}^{0} \equiv \mbox{\mathversion{bold}$e$}_{3} \).
System (\ref{arc}) can be solved explicitly to obtain
\begin{align*}
\mbox{\mathversion{bold}$v$}_{0}(s) 
= {}^{t}\big( \cos(s/R), -\sin(s/R), 0 \big).
\end{align*}
The above
\( \mbox{\mathversion{bold}$v$}_{0}\) corresponds to an arc-shaped 
filament, which is an explicit solution of problem (\ref{slant}). 
Namely, 
\begin{align*}
\mbox{\mathversion{bold}$v$}_{0}\times 
\mbox{\mathversion{bold}$v$}_{0ss} \equiv \mbox{\mathversion{bold}$0$}
\end{align*}
and hence, \( \mbox{\mathversion{bold}$v$}_{0}\) is a
stationary solution of problem (\ref{slant})
with \( \mbox{\mathversion{bold}$b$}={}^{t}(\cos (L/R), -\sin(L/R),0) \). 
We set 
\begin{align*}
\mbox{\mathversion{bold}$v$}^{R}(s,t) = \mbox{\mathversion{bold}$v$}^{R}(s)
:= {}^{t}\big( \cos(s/R), -\sin(s/R), 0 \big) 
\end{align*}
for 
\( s\in I_{L}\) and \( t>0 \) and refer to 
\( \mbox{\mathversion{bold}$v$}^{R}\) as the arc-shaped solution 
of problem (\ref{slant}). 
This shows that \( q_{R}\) corresponds to \( \mbox{\mathversion{bold}$v$}^{R}\) 
through the inverse generalized Hasimoto transformation.

\subsubsection{Stability of Arc-Shaped Solutions of Problem (\ref{slant})}
We derive stability estimates for arc-shaped solutions of problem (\ref{slant}).
These estimates are already derived by the author in \cite{27}, but we reiterate the 
statements and proofs for completeness.

For \( R>0\), \( \mbox{\mathversion{bold}$v$}^{R}\) is the solution of 
problem (\ref{slant}) with 
\begin{align*}
\mbox{\mathversion{bold}$v$}_{0}(s)
=
\mbox{\mathversion{bold}$v$}^{R}(s)
=
{}^{t}(\cos(s/R),-\sin (s/R), 0), \\[3mm]
\mbox{\mathversion{bold}$b$}={}^{t}(\cos(L/R),-\sin (L,/R), 0).
\end{align*}
Recall that \( \mbox{\mathversion{bold}$v$}^{R}\) is a stationary solution.
We consider perturbations of the arc-shaped solution 
\( \mbox{\mathversion{bold}$v$}^{R}\).
Let, \( \mbox{\mathversion{bold}$\varphi$}_{0}\in H^{3}(I_{L})\) 
be the initial perturbation satisfying the following.
\begin{as}
For the initial perturbation \( \mbox{\mathversion{bold}$\varphi$}_{0}\), 
we assume the following.
\begin{description}
\item[(A1)] \ \( \lvert \mbox{\mathversion{bold}$v$}_{0}(s)+\mbox{\mathversion{bold}$\varphi$}_{0}(s)\rvert =1 \)
for all \( s\in I_{L} \).
\item[(A2)] \ \( \mbox{\mathversion{bold}$v$}_{0}+\mbox{\mathversion{bold}$\varphi$}_{0}\)
and \( \mbox{\mathversion{bold}$b$}\) satisfy the compatibility conditions 
for problem {\rm (\ref{slant})} up to order
\( 1\).
\end{description}
\label{assump}
\end{as}
\noindent
Note that because \( \mbox{\mathversion{bold}$v$}_{0} \) and 
\( \mbox{\mathversion{bold}$b$} \)
satisfy the \( 0\)-th order compatibility condition, 
assumption (A2) implies \( \mbox{\mathversion{bold}$\varphi$}_{0}\rvert _{s=0}=\mbox{\mathversion{bold}$\varphi$}_{0}\rvert _{s=L}=\mbox{\mathversion{bold}$0$} \).
Let \( \mbox{\mathversion{bold}$v$}\in C\big( [0,\infty);H^{3}(I_{L})\big) \cap C^{1}\big( [0,\infty);H^{1}(I_{L})\big) \) be the solution of problem (\ref{slant}) with 
initial datum \( \mbox{\mathversion{bold}$v$}_{0}+\mbox{\mathversion{bold}$\varphi$}_{0}\)
and boundary datum \( \mbox{\mathversion{bold}$b$}\). For convenience, 
we refer to this solution \( \mbox{\mathversion{bold}$v$}\) as the perturbed
arc-shaped solution. Set
\( \mbox{\mathversion{bold}$\varphi$}=\mbox{\mathversion{bold}$v$}-\mbox{\mathversion{bold}$v$}^{R}\).
Then, \( \mbox{\mathversion{bold}$\varphi$}\) satisfies
\begin{align}
\left\{
\begin{array}{ll}
\mbox{\mathversion{bold}$\varphi$}_{t}=
\mbox{\mathversion{bold}$\varphi$}\times \mbox{\mathversion{bold}$\varphi$}_{ss}
+
\mbox{\mathversion{bold}$\varphi$}\times \mbox{\mathversion{bold}$v$}^{R}_{ss}
+
\mbox{\mathversion{bold}$v$}^{R}\times \mbox{\mathversion{bold}$\varphi$}_{ss}, &
s\in I_{L}, t>0, \\[3mm]
\mbox{\mathversion{bold}$\varphi$}(s,0) = \mbox{\mathversion{bold}$\varphi$}_{0}(s), &
s\in I_{L}, \\[3mm]
\mbox{\mathversion{bold}$\varphi$}(0,t)=\mbox{\mathversion{bold}$\varphi$}(L,t)=
\mbox{\mathversion{bold}$0$}. & \ 
\end{array}\right.
\label{perteq}
\end{align}
Note that
\( \mbox{\mathversion{bold}$\varphi$}\in C\big( [0,\infty);H^{3}(I_{L})\big) \cap C^{1}\big( [0,\infty);H^{1}(I_{L})\big) \).
We first prove the following lemma.
\begin{lm}
For \( \mbox{\mathversion{bold}$\varphi$}\in C\big( [0,\infty);H^{3}(I_{L})\big) \cap C^{1}\big( [0,\infty);H^{1}(I_{L})\big) \), set
\begin{align*}
E(\mbox{\mathversion{bold}$\varphi$}(t)):=
\| \mbox{\mathversion{bold}$\varphi$}_{s}(t)\|^{2}
-\frac{1}{R^{2}}\|\mbox{\mathversion{bold}$\varphi$}(t)\|^{2}.
\end{align*}
Then, if \( \mbox{\mathversion{bold}$\varphi$} \) satisfies 
{\rm (\ref{perteq})}, 
\( E(\mbox{\mathversion{bold}$\varphi$}(t)) = E(\mbox{\mathversion{bold}$\varphi$}_{0})\) for all \( t>0 \).
In other words, \( E\) is a conserved quantity. 
\label{conserved}
\end{lm}
\begin{proof}
From the first equation in (\ref{perteq}), we have
\begin{align*}
\frac{{\rm d}}{{\rm d}t}\| \mbox{\mathversion{bold}$\varphi$}\|^{2}
=
2(\mbox{\mathversion{bold}$\varphi$},\mbox{\mathversion{bold}$\varphi$}_{t})
=
2(\mbox{\mathversion{bold}$\varphi$}, \mbox{\mathversion{bold}$v$}^{R}
\times \mbox{\mathversion{bold}$\varphi$}_{ss})
=
-2(\mbox{\mathversion{bold}$\varphi$},
\mbox{\mathversion{bold}$v$}^{R}_{s}\times \mbox{\mathversion{bold}$\varphi$}_{s}),
\end{align*}
where integration by parts was used. Furthermore,
\begin{align*}
\frac{{\rm d}}{{\rm d}t}\|\mbox{\mathversion{bold}$\varphi$}_{s}\|^{2}
=
2(\mbox{\mathversion{bold}$\varphi$}_{s},\mbox{\mathversion{bold}$\varphi$}_{st})
=
-2(\mbox{\mathversion{bold}$\varphi$}_{ss},\mbox{\mathversion{bold}$\varphi$}_{t})
&=
-2(\mbox{\mathversion{bold}$\varphi$}_{ss},
\mbox{\mathversion{bold}$\varphi$}\times \mbox{\mathversion{bold}$v$}^{R}_{ss}) \\[3mm]
&= 
2(\mbox{\mathversion{bold}$\varphi$}_{s},\mbox{\mathversion{bold}$\varphi$}\times
\mbox{\mathversion{bold}$v$}^{R}_{sss}) \\[3mm]
&=
-\frac{2}{R^{2}}(\mbox{\mathversion{bold}$\varphi$},
\mbox{\mathversion{bold}$v$}^{R}_{s}\times \mbox{\mathversion{bold}$\varphi$}_{s}),
\end{align*}
where we substituted \( \mbox{\mathversion{bold}$v$}^{R}_{sss}=-\frac{1}{R^{2}}\mbox{\mathversion{bold}$v$}^{R}_{s}\).
Combining the two equalities show that 
\( \frac{{\rm d}}{{\rm d}t}E(\mbox{\mathversion{bold}$\varphi$}(t))=0\),
and this proves the lemma.
\end{proof}
From Lemma \ref{conserved}, the following Corollary immediately follows.
\begin{Cr}
For \( \mbox{\mathversion{bold}$\varphi$}\in C\big( [0,\infty);H^{3}(I_{L})\big) \cap C^{1}\big( [0,\infty);H^{1}(I_{L})\big) \) satisfying 
{\rm (\ref{perteq})}, if \( R> L/\pi\), there exists \( C>0\) such that
\begin{align*}
\| \mbox{\mathversion{bold}$\varphi$}(t)\|_{1}
\leq
C\| \mbox{\mathversion{bold}$\varphi$}_{0}\|_{1}
\end{align*}
holds for all \( t>0\), where \( C>0\) depends on \( L\) and \(R \).
\label{h1est}
\end{Cr}
\begin{proof}
Since \( \mbox{\mathversion{bold}$\varphi$}\rvert_{s=0}=\mbox{\mathversion{bold}$\varphi$}\rvert _{s=L} = \mbox{\mathversion{bold}$0$}\),
the Poincar\'e inequality
\begin{align*}
\| \mbox{\mathversion{bold}$\varphi$}(t)\|
\leq
\frac{L}{\pi}\| \mbox{\mathversion{bold}$\varphi$}_{s}(t)\|,
\end{align*}
is applicable and hence
\begin{align*}
E(\mbox{\mathversion{bold}$\varphi$}(t)) 
\geq 
\| \mbox{\mathversion{bold}$\varphi$}_{s}(t)\|^{2}
-
\frac{1}{R^{2}}
\left(
\frac{L}{\pi}
\right)^{2}
\|\mbox{\mathversion{bold}$\varphi$}_{s}(t)\|^{2}
=
c_{0}\|\mbox{\mathversion{bold}$\varphi$}_{s}(t)\|^{2}
\end{align*}
holds for all \( t>0\), where 
\( c_{0}= 1-\frac{L^{2}}{R^{2} \pi ^{2}} \in (0,1) \). This with 
Lemma \ref{conserved} yields
\begin{align*}
\|\mbox{\mathversion{bold}$\varphi$}_{s}(t)\|^{2}
\leq CE(\mbox{\mathversion{bold}$\varphi$}_{0})
\leq
C\|\mbox{\mathversion{bold}$\varphi$}_{0}\|_{1},
\end{align*}
where \( C>0\) depends on \( L \) and \( R\). 
Applying the Poincar\'e inequality once more, we have
\begin{align*}
\| \mbox{\mathversion{bold}$\varphi$}(t)\|_{1} 
\leq
C\|\mbox{\mathversion{bold}$\varphi$}_{0}\|_{1}
\end{align*}
for all \( t>0\).

\end{proof}
Next, we show higher-order estimates.
\begin{lm}
For \( \mbox{\mathversion{bold}$\varphi$}\in C\big( [0,\infty);H^{3}(I_{L})\big) \cap C^{1}\big( [0,\infty);H^{1}(I_{L})\big) \) satisfying {\rm (\ref{perteq})},
there exists \( C>0\) depending only on \( L\) and \(R\) such that
if \( R> L/\pi \), \( \mbox{\mathversion{bold}$\varphi$}\) satisfies
\begin{align*}
\|\mbox{\mathversion{bold}$\varphi$}_{ss}(t)\|_{1}
\leq
C\big( \| \mbox{\mathversion{bold}$\varphi$}_{0}\|_{3}
+
\|\mbox{\mathversion{bold}$\varphi$}_{0}\|_{3}^{3}\big)
\end{align*}
for all \( t>0\).

\label{high}
\end{lm}
\begin{proof}
We make use of conserved quantities for problem (\ref{slant}), which was 
also utilized in \cite{23}. Properties of solutions of
problem (\ref{slant}) that we make use of here are either well known, or proved in \cite{23}
and hence, we will use them without proof.
The perturbed arc-shaped solution 
\( \mbox{\mathversion{bold}$v$} \) satisfies
\( \lvert \mbox{\mathversion{bold}$v$}(s,t)\rvert =1\) for all \( s\in I_{L}\) and 
\( t>0\). Hence,
\begin{align*}
1=\lvert\mbox{\mathversion{bold}$v$}(s,t)\rvert^{2}
=
\lvert \mbox{\mathversion{bold}$v$}^{R}(s)
+\mbox{\mathversion{bold}$\varphi$}(s,t)\rvert ^{2}
\end{align*}
for all \( s\in I_{L}\) and \( t>0\).
Since \( \lvert \mbox{\mathversion{bold}$v$}^{R}(s)\rvert =1\), we have
\begin{align}
2\mbox{\mathversion{bold}$v$}^{R}(s)\cdot \mbox{\mathversion{bold}$\varphi$}(s,t)
=
-\lvert\mbox{\mathversion{bold}$\varphi$}(s,t)\rvert^{2}
\label{nostr}
\end{align}
for all \( s\in I_{L}\) and \( t>0\).
We also have the following conserved quantities.
\begin{align}
&\frac{{\rm d}}{{\rm d}t}
\bigg\{
\|\mbox{\mathversion{bold}$v$}_{ss}\|^{2}
-\frac{5}{4}\| \lvert \mbox{\mathversion{bold}$v$}_{s}\rvert ^{2} \|^{2} 
\bigg\} = 0,  \label{conserve1} \\[3mm]
&\frac{{\rm d}}{{\rm d}t}
\bigg\{
\| \mbox{\mathversion{bold}$v$}_{sss}\|^{2}
-
\frac{7}{2}\| \lvert \mbox{\mathversion{bold}$v$}_{s}\rvert
\lvert \mbox{\mathversion{bold}$v$}_{ss}\rvert \|^{2}
-
14\| \mbox{\mathversion{bold}$v$}_{s}\cdot \mbox{\mathversion{bold}$v$}_{ss}\|^{2}
+
\frac{21}{8}\| \lvert \mbox{\mathversion{bold}$v$}_{s}\rvert ^{3} \|^{2}
\bigg\}=0.
\label{conserve2}
\end{align}
These quantities are well known, and the above equality can be verified by 
direct calculation.

First we set
\begin{align*}
E_{1}(\mbox{\mathversion{bold}$v$}(t))
:=
\|\mbox{\mathversion{bold}$v$}_{ss}(t)\|^{2}
-\frac{5}{4}\| \lvert \mbox{\mathversion{bold}$v$}_{s}(t)\rvert ^{2} \|^{2}.
\end{align*}
For \( t\geq 0\) we substitute 
\( \mbox{\mathversion{bold}$v$}=\mbox{\mathversion{bold}$v$}^{R}+\mbox{\mathversion{bold}$\varphi$}\) and decompose as follows.
\begin{align*}
E_{1}(\mbox{\mathversion{bold}$v$}(t))
=
E_{1}(\mbox{\mathversion{bold}$v$}^{R})
+
E_{1}(\mbox{\mathversion{bold}$\varphi$}(t))
+
R_{1}(\mbox{\mathversion{bold}$v$}^{R},\mbox{\mathversion{bold}$\varphi$}(t)),
\end{align*}
where \( R_{1}(\mbox{\mathversion{bold}$v$}^{R},\mbox{\mathversion{bold}$\varphi$}(t))\)
is given by
\begin{align*}
R_{1}(\mbox{\mathversion{bold}$v$}^{R},\mbox{\mathversion{bold}$\varphi$}(t))
&=
2(\mbox{\mathversion{bold}$v$}^{R}_{ss},\mbox{\mathversion{bold}$\varphi$}_{ss})
-
5(\lvert \mbox{\mathversion{bold}$v$}^{R}_{s}\rvert ^{2}
\mbox{\mathversion{bold}$v$}^{R}_{s},
\mbox{\mathversion{bold}$\varphi$}^{R}_{s})
-
5((\mbox{\mathversion{bold}$v$}^{R}_{s}\cdot \mbox{\mathversion{bold}$\varphi$}_{s})
\mbox{\mathversion{bold}$v$}^{R}_{s},\mbox{\mathversion{bold}$\varphi$}_{s}) \\[3mm]
& \qquad -
\frac{5}{2}(\lvert \mbox{\mathversion{bold}$v$}^{R}_{s}\rvert ^{2}
\mbox{\mathversion{bold}$\varphi$}_{s},\mbox{\mathversion{bold}$\varphi$}_{s})
-
5(\lvert \mbox{\mathversion{bold}$\varphi$}_{s}\rvert ^{2}
\mbox{\mathversion{bold}$v$}^{R}_{s},
\mbox{\mathversion{bold}$\varphi$}_{s})
\end{align*}
On the other hand, since \( E_{1}\) is conserved, we have
\begin{align*}
E_{1}(\mbox{\mathversion{bold}$v$}(t))
=
E_{1}(\mbox{\mathversion{bold}$v$}_{0})
&=
E_{1}(\mbox{\mathversion{bold}$v$}^{R}+\mbox{\mathversion{bold}$\varphi$}_{0})\\[3mm]
&=E_{1}(\mbox{\mathversion{bold}$v$}^{R}) + E_{1}(\mbox{\mathversion{bold}$\varphi$}_{0})
+
R_{1}(\mbox{\mathversion{bold}$v$}^{R},\mbox{\mathversion{bold}$\varphi$}_{0}),
\end{align*}
and combining the two equalities, we have
\begin{align*}
E_{1}(\mbox{\mathversion{bold}$\varphi$}(t))
=
E_{1}(\mbox{\mathversion{bold}$\varphi$}_{0})
-
R_{1}(\mbox{\mathversion{bold}$v$}^{R},\mbox{\mathversion{bold}$\varphi$})
+
R_{1}(\mbox{\mathversion{bold}$v$}^{R},\mbox{\mathversion{bold}$\varphi$}_{0}).
\end{align*}
We estimate \( R_{1}(\mbox{\mathversion{bold}$v$}^{R},\mbox{\mathversion{bold}$\varphi$}) \)
in detail and omit the details for 
\( R_{1}(\mbox{\mathversion{bold}$v$}^{R},\mbox{\mathversion{bold}$\varphi$}_{0})\)
since they are the same.
The terms that are linear with respect to
\( \mbox{\mathversion{bold}$\varphi$}\) can be estimated as follows.
First we have
\begin{align*}
2(\mbox{\mathversion{bold}$v$}^{R}_{ss},\mbox{\mathversion{bold}$\varphi$}_{ss})
=
-\frac{2}{R^{2}}(\mbox{\mathversion{bold}$v$}^{R},
\mbox{\mathversion{bold}$\varphi$}_{ss})
&=
\frac{2}{R^{2}}(\mbox{\mathversion{bold}$v$}^{R}_{s},
\mbox{\mathversion{bold}$\varphi$}_{s})
-\frac{2}{R^{2}}
\big[ \mbox{\mathversion{bold}$v$}^{R}\cdot 
\mbox{\mathversion{bold}$\varphi$}_{s}\big]^{L}_{s=0}\\[3mm]
&=
-\frac{2}{R^{2}}(\mbox{\mathversion{bold}$v$}^{R}_{ss},\mbox{\mathversion{bold}$\varphi$})
-\frac{2}{R^{2}}
\big[ \mbox{\mathversion{bold}$v$}^{R}\cdot 
\mbox{\mathversion{bold}$\varphi$}_{s}\big]^{L}_{s=0} \\[3mm]
&=
\frac{2}{R^{4}}(\mbox{\mathversion{bold}$v$}^{R},\mbox{\mathversion{bold}$\varphi$})
-\frac{2}{R^{2}}
\big[ \mbox{\mathversion{bold}$v$}^{R}\cdot 
\mbox{\mathversion{bold}$\varphi$}_{s}\big]^{L}_{s=0}.
\end{align*}
Integrating equation (\ref{nostr}) with respect to \( s\) shows that
\( (\mbox{\mathversion{bold}$v$}^{R},\mbox{\mathversion{bold}$\varphi$})= -\frac{1}{2}\|\mbox{\mathversion{bold}$\varphi$}\|^{2}\).
Furthermore, differentiating equation (\ref{nostr}) with respect to \( s\) 
and taking the trace \( s=0,L\) yields
\( \mbox{\mathversion{bold}$v$}^{R}\cdot \mbox{\mathversion{bold}$\varphi$}_{s}\rvert_{s=0,L} = 0\).
Hence we have
\begin{align*}
2(\mbox{\mathversion{bold}$v$}^{R}_{ss},\mbox{\mathversion{bold}$\varphi$}_{ss})
=
-\frac{1}{R^{4}}\|\mbox{\mathversion{bold}$\varphi$}\|^{2}.
\end{align*}
Similarly, we have
\begin{align*}
-5(\lvert \mbox{\mathversion{bold}$v$}^{R}_{s}\rvert^{2}
\mbox{\mathversion{bold}$v$}^{R}_{s},
\mbox{\mathversion{bold}$\varphi$}_{s})
=
-\frac{5}{R^{2}}(\mbox{\mathversion{bold}$v$}^{R}_{s},\mbox{\mathversion{bold}$\varphi$}_{s})
=
\frac{5}{2R^{4}}\|\mbox{\mathversion{bold}$\varphi$}_{s}\|^{2},
\end{align*}
where \( \lvert\mbox{\mathversion{bold}$v$}^{R}_{s}\rvert \equiv \frac{1}{R}\) was 
substituted. 
Taking into account the above two equalities, we have
\begin{align*}
\lvert 2(\mbox{\mathversion{bold}$v$}^{R}_{ss},
\mbox{\mathversion{bold}$\varphi$}_{ss})
-5(\lvert \mbox{\mathversion{bold}$v$}^{R}_{s}\rvert^{2}
\mbox{\mathversion{bold}$v$}^{R}_{s},
\mbox{\mathversion{bold}$\varphi$}_{s}) \rvert
\leq
C\| \mbox{\mathversion{bold}$\varphi$}\|_{1}^{2}
\leq
C\| \mbox{\mathversion{bold}$\varphi$}_{0}\|_{1}^{2},
\end{align*}
where Corollary \ref{h1est} was applied in the last inequality.

For the other terms, we have
\begin{align*}
\big\lvert -
5((\mbox{\mathversion{bold}$v$}^{R}_{s}\cdot \mbox{\mathversion{bold}$\varphi$}_{s})
\mbox{\mathversion{bold}$v$}^{R}_{s},\mbox{\mathversion{bold}$\varphi$}_{s})
-
\frac{5}{2}(\lvert \mbox{\mathversion{bold}$v$}^{R}_{s}\rvert ^{2}
\mbox{\mathversion{bold}$\varphi$}_{s},\mbox{\mathversion{bold}$\varphi$}_{s})
&-
5(\lvert \mbox{\mathversion{bold}$\varphi$}_{s}\rvert ^{2}
\mbox{\mathversion{bold}$v$}^{R}_{s},
\mbox{\mathversion{bold}$\varphi$}_{s}) \big\rvert \\[3mm]
&\leq
C\big( 
\|\mbox{\mathversion{bold}$\varphi$}_{s}\|^{2}
+
\| \mbox{\mathversion{bold}$\varphi$}_{s} \|^{2}
\|\mbox{\mathversion{bold}$\varphi$}_{s}\|_{1}
\big) \\[3mm]
&\leq
C\big( \|\mbox{\mathversion{bold}$\varphi$}_{0}\|_{1}^{2}
+
\| \mbox{\mathversion{bold}$\varphi$}_{0}\|_{1}^{3}
\big)
\end{align*}
where \( C>0\) depends only on \( R \). 
Corollary \ref{h1est} and the interpolation inequality yield
\begin{align*}
E_{1}(\mbox{\mathversion{bold}$\varphi$}(t))
&\geq 
\|\mbox{\mathversion{bold}$\varphi$}_{ss}\|^{2}
-C\| \mbox{\mathversion{bold}$\varphi$}_{ss}\|
\| \mbox{\mathversion{bold}$\varphi$}_{s}\|^{3}
-
C\|\mbox{\mathversion{bold}$\varphi$}_{s}\|^{2} \\[3mm]
&\geq
c_{1}\|\mbox{\mathversion{bold}$\varphi$}_{ss}\|^{2}
-C\|\mbox{\mathversion{bold}$\varphi$}_{s}\|^{2}
-C\|\mbox{\mathversion{bold}$\varphi$}_{s}\|^{6} \\[3mm]
&\geq
c_{1}\|\mbox{\mathversion{bold}$\varphi$}_{ss}\|^{2}
-
C\|\mbox{\mathversion{bold}$\varphi$}_{0}\|_{1}^{2}
-
C\| \mbox{\mathversion{bold}$\varphi$}_{0}\|_{1}^{6}
\end{align*}
where \( c_{1}>0\) depends only on \( L\) and \( C>0\) depends on \( L\)
and \( R \). Combining all of the obtained estimates yields
\begin{align}
\| \mbox{\mathversion{bold}$\varphi$}_{ss}(t)\|^{2}
\leq
C\big(
\| \mbox{\mathversion{bold}$\varphi$}_{0}\|_{2}^{2}
+
\|\mbox{\mathversion{bold}$\varphi$}_{0}\|_{2}^{6}
\big)
\label{h2estimate}
\end{align}
for all \( t>0\). Here, \( C>0\) depends on \( L\) and \( R\).

Similarly, if we set
\begin{align*}
E_{2}(\mbox{\mathversion{bold}$v$}(t))
=
\| \mbox{\mathversion{bold}$v$}_{sss}(t)\|^{2}
-
\frac{7}{2}\| \lvert \mbox{\mathversion{bold}$v$}_{s}(t)\rvert
\lvert \mbox{\mathversion{bold}$v$}_{ss}(t)\rvert \|^{2}
-
14\| \mbox{\mathversion{bold}$v$}_{s}(t)\cdot \mbox{\mathversion{bold}$v$}_{ss}(t)\|^{2}
+
\frac{21}{8}\| \lvert \mbox{\mathversion{bold}$v$}_{s}(t)\rvert ^{3} \|^{2},
\end{align*}
we have
\begin{align*}
E_{2}(\mbox{\mathversion{bold}$v$}^{R}+\mbox{\mathversion{bold}$\varphi$}_{0})
&=E_{2}(\mbox{\mathversion{bold}$v$}^{R})+E_{2}(\mbox{\mathversion{bold}$\varphi$}_{0})
+
R_{2}(\mbox{\mathversion{bold}$v$}^{R},\mbox{\mathversion{bold}$\varphi$}_{0}), \\[3mm]
E_{2}(\mbox{\mathversion{bold}$v$}^{R}+\mbox{\mathversion{bold}$\varphi$}(t))
&=
E_{2}(\mbox{\mathversion{bold}$v$}^{R})+E_{2}(\mbox{\mathversion{bold}$\varphi$}(t))
+
R_{2}(\mbox{\mathversion{bold}$v$}^{R},\mbox{\mathversion{bold}$\varphi$}(t)),
\end{align*}
and the conservation of \(  E_{2}\) implies that the above two are equal.
Hence,
\begin{align*}
E_{2}(\mbox{\mathversion{bold}$\varphi$}(t))
=
E_{2}(\mbox{\mathversion{bold}$\varphi$}_{0})+R_{2}(\mbox{\mathversion{bold}$v$}^{R},
\mbox{\mathversion{bold}$\varphi$}_{0})-R_{2}(\mbox{\mathversion{bold}$v$}^{R},
\mbox{\mathversion{bold}$\varphi$}(t)).
\end{align*}
Here, 
\begin{align*}
R_{2}(\mbox{\mathversion{bold}$v$}^{R},\mbox{\mathversion{bold}$\varphi$}(t))
&=
2(\mbox{\mathversion{bold}$v$}^{R}_{sss},\mbox{\mathversion{bold}$\varphi$}_{sss})
+
2(\lvert \mbox{\mathversion{bold}$v$}^{R}_{s}\rvert ^{2}
\mbox{\mathversion{bold}$v$}^{R}_{ss},
\mbox{\mathversion{bold}$\varphi$}_{ss})
+
(\lvert \mbox{\mathversion{bold}$v$}^{R}_{ss}\rvert^{2}
\mbox{\mathversion{bold}$v$}^{R}_{s},
\mbox{\mathversion{bold}$\varphi$}_{s}) \\[3mm]
& \qquad +
8(\lvert\mbox{\mathversion{bold}$v$}^{R}_{s}\rvert^{4}
\mbox{\mathversion{bold}$v$}^{R}_{s},
\mbox{\mathversion{bold}$\varphi$}_{s})
+
N(\mbox{\mathversion{bold}$v$}^{R},\mbox{\mathversion{bold}$\varphi$}(t)),
\end{align*}
where \( N(\mbox{\mathversion{bold}$v$}^{R},\mbox{\mathversion{bold}$\varphi$}(t)) \) are
terms that are nonlinear with respect to \( \mbox{\mathversion{bold}$\varphi$}\).
We estimate as follows. First we have
\begin{align*}
2(\mbox{\mathversion{bold}$v$}^{R}_{sss},\mbox{\mathversion{bold}$\varphi$}_{sss})
=
-\frac{2}{R^{2}}(\mbox{\mathversion{bold}$v$}^{R}_{s},\mbox{\mathversion{bold}$\varphi$}_{sss})
=
-\frac{2}{R^{2}}[ \mbox{\mathversion{bold}$v$}^{R}_{s}\cdot
 \mbox{\mathversion{bold}$\varphi$}_{ss}]^{L}_{s=0}
+\frac{2}{R^{2}}(\mbox{\mathversion{bold}$v$}^{R}_{ss},\mbox{\mathversion{bold}$\varphi$}_{ss})
\end{align*}
Taking the exterior product of 
\( \mbox{\mathversion{bold}$v$}^{R}_{s}\) from the left with the first equation in
problem (\ref{perteq}), we have
\begin{align}
\mbox{\mathversion{bold}$v$}^{R}_{s}\times \mbox{\mathversion{bold}$\varphi$}_{t}
=
\mbox{\mathversion{bold}$v$}^{R}_{s}\times
(\mbox{\mathversion{bold}$\varphi$}\times \mbox{\mathversion{bold}$\varphi$}_{ss})
+
\mbox{\mathversion{bold}$v$}^{R}_{s}\times 
(\mbox{\mathversion{bold}$\varphi$}\times \mbox{\mathversion{bold}$v$}^{R}_{ss})
+
\mbox{\mathversion{bold}$v$}^{R}_{s}\times
(\mbox{\mathversion{bold}$v$}^{R}\times \mbox{\mathversion{bold}$\varphi$}_{ss}).
\label{trace}
\end{align}
Since \( \mbox{\mathversion{bold}$\varphi$}_{t}\rvert_{s=0,L}=\mbox{\mathversion{bold}$0$}\),
taking the trace at \( s=0,L\) in (\ref{trace}) yields
\begin{align*}
(\mbox{\mathversion{bold}$v$}^{R}_{s}\cdot \mbox{\mathversion{bold}$\varphi$}_{ss})
\mbox{\mathversion{bold}$v$}^{R}\rvert _{s=0,L}=\mbox{\mathversion{bold}$0$}.
\end{align*}
\( \lvert\mbox{\mathversion{bold}$v$}^{R}\rvert \equiv 1\) further implies that
\( \mbox{\mathversion{bold}$v$}^{R}_{s}\cdot \mbox{\mathversion{bold}$\varphi$}_{ss}\rvert_{s=0,L}=0\). Hence, we have
\begin{align*}
2(\mbox{\mathversion{bold}$v$}^{R}_{sss},\mbox{\mathversion{bold}$\varphi$}_{sss})
=
\frac{2}{R^{2}}(\mbox{\mathversion{bold}$v$}^{R}_{ss},\mbox{\mathversion{bold}$\varphi$}_{ss})
=
-\frac{1}{R^{6}}\| \mbox{\mathversion{bold}$\varphi$} \|^{2}.
\end{align*}
Next we estimate
\begin{align*}
2(\lvert\mbox{\mathversion{bold}$v$}^{R}_{s}\rvert ^{2}
\mbox{\mathversion{bold}$v$}^{R}_{ss},
\mbox{\mathversion{bold}$\varphi$}_{ss})
&=
\frac{2}{R^{2}}(\mbox{\mathversion{bold}$v$}^{R}_{ss},
\mbox{\mathversion{bold}$\varphi$}_{ss})
=
-\frac{1}{R^{6}}\|\mbox{\mathversion{bold}$\varphi$}\|^{2}, \\[3mm]
(\lvert \mbox{\mathversion{bold}$v$}^{R}_{ss}\rvert^{2}
\mbox{\mathversion{bold}$v$}^{R}_{s},
\mbox{\mathversion{bold}$\varphi$}_{s}) 
+
8(\lvert\mbox{\mathversion{bold}$v$}^{R}_{s}\rvert^{4}
\mbox{\mathversion{bold}$v$}^{R}_{s},
\mbox{\mathversion{bold}$\varphi$}_{s})
&=
\frac{9}{R^{4}}(\mbox{\mathversion{bold}$v$}^{R}_{s},
\mbox{\mathversion{bold}$\varphi$}_{s})
=
-\frac{9}{2R^{4}}\|\mbox{\mathversion{bold}$\varphi$}_{s}\|^{2},
\end{align*}
and finally,
\begin{align*}
\lvert N(\mbox{\mathversion{bold}$v$}^{R},
\mbox{\mathversion{bold}$\varphi$}(t)) \rvert
\leq
C\big(
\| \mbox{\mathversion{bold}$\varphi$}\|_{2}^{2}
+
\| \mbox{\mathversion{bold}$\varphi$}\|_{2}^{5}
\big).
\end{align*}
Additionally, 
\begin{align*}
E_{2}( \mbox{\mathversion{bold}$\varphi$}(t) )
\geq 
\|\mbox{\mathversion{bold}$\varphi$}_{sss}(t)\|^{2}
-
C\big(
\| \mbox{\mathversion{bold}$\varphi$}(t)\|_{2}^{2}
+
\| \mbox{\mathversion{bold}$\varphi$}(t)\|_{2}^{6}
\big)
\end{align*}
along with all the estimates previously obtained shows that
\begin{align}
\| \mbox{\mathversion{bold}$\varphi$}_{sss}(t) \|^{2}
\leq
C\big(
\| \mbox{\mathversion{bold}$\varphi$}_{0}\|_{3}^{2}
+
\|\mbox{\mathversion{bold}$\varphi$}_{0}\|_{3}^{6}
\big)
\label{h3estimate}
\end{align}
holds for all \( t>0\), where \( C>0\) depends on \( L \) and \( R\).
Estimates (\ref{h2estimate}) and (\ref{h3estimate}) 
combined proves Lemma \ref{high}.
\end{proof}

In summary, we have the following.
\begin{pr}
For \( R> L/\pi \), the arc-shaped solution 
\( \mbox{\mathversion{bold}$v$}^{R} \) of problem {\rm (\ref{slant})} is 
stable in the following sense.
For \( \mbox{\mathversion{bold}$\varphi$}_{0}\in H^{3}(I_{L})\)
satisfying Assumption {\rm \ref{assump}}, the perturbed arc-shaped solution
\( \mbox{\mathversion{bold}$v$}\in C\big( [0,\infty);H^{3}(I_{L})\big) \cap C^{1}\big( [0,\infty);H^{1}(I_{L})\big)\) satisfies
\begin{align*}
\| \mbox{\mathversion{bold}$v$}(t)-\mbox{\mathversion{bold}$v$}^{R}\|_{3}
\leq
C\big( 
\| \mbox{\mathversion{bold}$\varphi$}_{0}\|_{3}
+
\| \mbox{\mathversion{bold}$\varphi$}_{0}\|_{3}^{3}
\big)
\end{align*}
for all \( t>0\). Here, \( C>0\) depends on \( L\) and \(R\).

\label{arcstable}
\end{pr}


\subsubsection{Stability Estimates and the generalized Hasimoto transformation}

In this section, we show that stability estimates for 
solutions of problem (\ref{slant}) transfer to stability
estimates for solutions of problem (\ref{plane})
through the generalized Hasimoto transformation.

Let, \( \mbox{\mathversion{bold}$v$}_{1}\) and 
\( \mbox{\mathversion{bold}$v$}_{2}\) be 
solutions of problem (\ref{slant}) belonging to
\( C\big( [0,\infty);H^{3}(I_{L})\big) \cap C^{1}\big( [0,\infty);\)
\( H^{1}(I_{L})\big) \)
with possibly different initial datum, but the same 
boundary datum. Additionally, 
suppose that there exists \( C_{0} >0\) independent of
\( \mbox{\mathversion{bold}$v$}_{1}\) and 
\( \mbox{\mathversion{bold}$v$}_{2}\) and \( M \geq 0\) such that
\begin{align}
\| \mbox{\mathversion{bold}$v$}_{1}(t)
-\mbox{\mathversion{bold}$v$}_{2}(t)\|_{3}
\leq
C_{0} M
\label{hasiv}
\end{align}
holds for all \( t>0\).
Define \( \mbox{\mathversion{bold}$e$}^{1}\) and
\( \mbox{\mathversion{bold}$e$}^{2}\) as the solution 
of
\begin{align}
\left\{
\begin{array}{ll}
\mbox{\mathversion{bold}$e$}^{1}_{s}
=-(\mbox{\mathversion{bold}$v$}_{1s}\cdot
\mbox{\mathversion{bold}$e$}^{1})\mbox{\mathversion{bold}$v$}_{1}, 
& s\in I_{L}, \\[3mm]
\mbox{\mathversion{bold}$e$}^{1}\rvert_{s=0}=\mbox{\mathversion{bold}$e$}_{2}, & \
\end{array}\right. \label{e1}
\end{align}
and
\begin{align}
\left\{
\begin{array}{ll}
\mbox{\mathversion{bold}$e$}^{2}_{s}
=-(\mbox{\mathversion{bold}$v$}_{2s}\cdot
\mbox{\mathversion{bold}$e$}^{2})\mbox{\mathversion{bold}$v$}_{2}, 
& s\in I_{L}, \\[3mm]
\mbox{\mathversion{bold}$e$}^{2}\rvert_{s=0}=\mbox{\mathversion{bold}$e$}_{2}, & \
\end{array}\right.
\label{e2}
\end{align}
respectively, and define \( \mbox{\mathversion{bold}$w$}^{1}\) and
\( \mbox{\mathversion{bold}$w$}^{2}\) by
\begin{align}
\mbox{\mathversion{bold}$w$}^{1}
=\mbox{\mathversion{bold}$v$}_{1}\times \mbox{\mathversion{bold}$e$}^{1}
\qquad \text{and} \qquad
\mbox{\mathversion{bold}$w$}^{2}
=\mbox{\mathversion{bold}$v$}_{2}\times \mbox{\mathversion{bold}$e$}^{2}.
\label{w12}
\end{align}
Again, note that \( \mbox{\mathversion{bold}$e$}^{i}\) and
\( \mbox{\mathversion{bold}$w$}^{i}\) depend on \( t\) 
through \( \mbox{\mathversion{bold}$v$}_{i}\).
Similarly to Section \ref{inverse hasimoto}, 
estimating \( \mbox{\mathversion{bold}$e$}^{1}-\mbox{\mathversion{bold}$e$}^{2}\) and 
\( \mbox{\mathversion{bold}$w$}^{1}-\mbox{\mathversion{bold}$w$}^{2}\) using 
(\ref{e1}), (\ref{e2}), and (\ref{w12}) 
yields
\begin{align}
\|
\mbox{\mathversion{bold}$e$}^{1}(t)
-\mbox{\mathversion{bold}$e$}^{2}(t)
\|_{3}
&\leq
CM, \label{hasie} \\[3mm]
\|
\mbox{\mathversion{bold}$w$}^{1}(t)
-\mbox{\mathversion{bold}$w$}^{2}(t)
\|_{3}
&\leq
CM, \label{hasiw}
\end{align}
for \( t>0\).
Here, \( C>0\) depends on \( C_{0}\) and
\( \|\mbox{\mathversion{bold}$v$}_{1s}\|_{2}+\|\mbox{\mathversion{bold}$v$}_{2s}\|_{2}\) and is non-decreasing with respect to 
\( C_{0}>0\) and
\( \|\mbox{\mathversion{bold}$v$}_{1s}\|_{2}+\|\mbox{\mathversion{bold}$v$}_{2s}\|_{2}\)

Since \( \lvert\mbox{\mathversion{bold}$v$}_{i}\rvert^{2}\equiv 1\) \( (i=1,2)\),
we have the decomposition
\begin{align}
\mbox{\mathversion{bold}$v$}_{1s}
=
\psi^{1}_{1}\mbox{\mathversion{bold}$e$}^{1}
+
\psi^{1}_{2}\mbox{\mathversion{bold}$w$}^{1}, \label{hasiest} \\[3mm]
\mbox{\mathversion{bold}$v$}_{2s}
=
\psi^{2}_{1}\mbox{\mathversion{bold}$e$}^{2}
+
\psi^{2}_{2}\mbox{\mathversion{bold}$w$}^{2}, \label{hasiest2}
\end{align}
for some \( \psi^{i}_{j}(s,t)\) \( (i,j=1,2)\).
Furthermore, from equations (\ref{hasiest}) and (\ref{hasiest2}), 
along with (\ref{e1}), (\ref{e2}), and (\ref{w12}),
we have
\begin{align}
\|\mbox{\mathversion{bold}$v$}_{1s}(t)\|_{2}+
\|\mbox{\mathversion{bold}$v$}_{2s}(t)\|_{2}
\leq
C\big( \sum_{i,j=1}^{2} \| \psi^{i}_{j}(t)\|_{2}
+
\sum_{i,j=1}^{2} \| \psi^{i}_{j}(t)\|_{2}^{3}\big),
\label{vphi}
\end{align}
where \( C>0\) is independent of \( \psi^{i}_{j} \ (i,j=1,2)\).
Substituting (\ref{hasiest}) and (\ref{hasiest2}) into the 
first equation of (\ref{e1}) and (\ref{e2}) yields
%
%
\begin{align*}
(\psi^{1}_{1}-\psi^{2}_{1})
\mbox{\mathversion{bold}$v$}_{1}
&=
-(\mbox{\mathversion{bold}$e$}^{1}
-
\mbox{\mathversion{bold}$e$}^{2})_{s}
+
\psi^{2}_{1}
(\mbox{\mathversion{bold}$v$}_{2}-\mbox{\mathversion{bold}$v$}_{1}), \\[3mm]
(\psi^{1}_{2}-\psi^{2}_{2})
\mbox{\mathversion{bold}$v$}_{1}
&=
-(\mbox{\mathversion{bold}$w$}^{1}-\mbox{\mathversion{bold}$w$}^{2})_{s}
+
\psi^{2}_{2}
(\mbox{\mathversion{bold}$v$}_{2}-\mbox{\mathversion{bold}$v$}_{1}),
\end{align*}
and taking the inner product of the above two equations
with \( \mbox{\mathversion{bold}$v$}_{1}\) yields
\begin{align}
(\psi^{1}_{1}-\psi^{2}_{1})
&=
-(\mbox{\mathversion{bold}$e$}^{1}
-
\mbox{\mathversion{bold}$e$}^{2})_{s}
\cdot \mbox{\mathversion{bold}$v$}_{1}
+
\psi^{2}_{1}
\big( (\mbox{\mathversion{bold}$v$}_{2}-\mbox{\mathversion{bold}$v$}_{1}) 
\cdot \mbox{\mathversion{bold}$v$}_{1}\big), \label{hasiest3} \\[3mm]
(\psi^{1}_{2}-\psi^{2}_{2})
&=
-(\mbox{\mathversion{bold}$w$}^{1}-\mbox{\mathversion{bold}$w$}^{2})_{s}
\cdot \mbox{\mathversion{bold}$v$}_{1}
+
\psi^{2}_{2}
\big( (\mbox{\mathversion{bold}$v$}_{2}-\mbox{\mathversion{bold}$v$}_{1})
\cdot \mbox{\mathversion{bold}$v$}_{1} \big). \label{hasiest4}
\end{align}
We see from direct estimates based on equations (\ref{hasiest3}) and 
(\ref{hasiest4}) that
\begin{align*}
\| \psi^{1}_{1}(t)-\psi^{2}_{1}(t) \|_{2} \leq C_{1}M, \\[3mm]
\| \psi^{1}_{2}(t)-\psi^{2}_{2}(t) \|_{2} \leq C_{1}M
\end{align*}
holds for all \( t>0\), where \( C_{1}>0\) depends on \( C_{0}\) and
\( \sum_{i,j=1}^{2} \| \psi^{i}_{j}(t)\|_{2}\) and is 
non-decreasing with respect to \( C_{0}>0\) and
\( \sum_{i,j=1}^{2} \| \psi^{i}_{j}(t)\|_{2}\).
Here, (\ref{hasie}), (\ref{hasiw}), and (\ref{vphi}) was also utilized.

Estimate (\ref{energy}) along with estimate (\ref{psiest}) implies that there exists \( C_{\ast}>0\) such that
\(
\sum_{i,j=1}^{2} \| \psi^{i}_{j}(t)\|_{2}\leq C_{\ast}
\)
holds for all \( t>0\). Here, \( C_{\ast}=C_{\ast}(K) \) depends on 
\( K=\| \mbox{\mathversion{bold}$v$}_{1}(0)\|_{3}+\|\mbox{\mathversion{bold}$v$}_{2}(0)\|_{3}\) and is
non-decreasing with respect to \( K>0 \). 
We briefly remark that
estimate (\ref{energy}) was utilized in
\cite{23} and is derived
from the conserved quantities (\ref{conserve1}), (\ref{conserve2}), and
\begin{align*}
\frac{{\rm d}}{{\rm d}t}\|\mbox{\mathversion{bold}$v$}_{s}(t)\|
=0,
\end{align*}
along with the fact that \( \lvert \mbox{\mathversion{bold}$v$}(s,t)\rvert =1 \) for all \( s\in I_{L}\) and \( t>0\).
Setting \( \psi_{1}(s,t):= \psi^{1}_{1}(s,t)+{\rm i}\psi^{1}_{2}(s,t) \) 
and \( \psi_{2}(s,t):=\psi^{2}_{1}(s,t) + {\rm i}\psi^{2}_{2}(s,t) \), we 
see that
\begin{align*}
\| \psi_{1}(t)-\psi_{2}(t)\|_{2}
\leq
C_{2}M
\end{align*}
holds for all \( t>0\), where \( C_{2}>0\) depends on \( C_{0}\) and 
\( K\) and is non-decreasing with respect to \( C_{0}>0 \) and \( K>0\).
Finally, setting
\begin{align*}
q_{1}(s,t)&:=\psi_{1}(s,t)
\exp \left\{ \frac{{\rm i}}{2} \int^{t}_{0}\lvert\psi_{1}(0,\tau )\rvert^{2}
 \ d\tau \right\}, \\[3mm]
q_{2}(s,t)&:= \psi_{2}(s,t)\exp \left\{ \frac{{\rm i}}{2} \int^{t}_{0}
\lvert \psi_{2}(0,\tau )\rvert ^{2} \ d\tau \right\},
\end{align*}
we have
\begin{align*}
\| \psi_{1}(t)-\psi_{2}(t)\|_{2}
&=
\left\| \exp \left\{ \frac{{\rm i}}{2} \int^{t}_{0}
\big( \lvert \psi_{2}(0,\tau )\rvert ^{2}
-\lvert \psi_{1}(0,\tau )\rvert^{2}\big)
 \ d\tau \right\}
 q_{1}(t)-q_{2}(t) \right\|_{2} \\[3mm]
& \geq
\inf_{\theta\in \mathbf{R}}\|\exp\{ {\rm i}\theta \}q_{1}(t)
-q_{2}(t)\|_{2}
\end{align*}
for all \( t>0\). 
In summary, we have
\begin{pr}
Let \( \mbox{\mathversion{bold}$v$}_{1},\mbox{\mathversion{bold}$v$}_{2}\in C\big([0,\infty);H^{3}(I_{L})\big)\cap C^{1}\big([0,\infty);H^{1}(I_{L})\big) \)
be solutions of problem {\rm (\ref{slant})} with 
initial datum \( \mbox{\mathversion{bold}$v$}_{0,1},\mbox{\mathversion{bold}$v$}_{0,2}\in H^{3}(I_{L}) \), respectively,
and a common boundary datum
\( \mbox{\mathversion{bold}$b$}\in \mathbf{S}^{2}\).
Let \( q_{1},q_{2}\in C\big([0,\infty);H^{2}(I_{L})\big)\cap C^{1}\big( [0,\infty);L^{2}(I_{L})\big) \) be solutions of
problem {\rm (\ref{plane})} corresponding to 
\( \mbox{\mathversion{bold}$v$}_{1}\) and \( \mbox{\mathversion{bold}$v$}_{2}\)
respectively, through the generalized Hasimoto transformation.
Furthermore, assume that there exists 
\( C_{0}>0\) independent of 
\( \mbox{\mathversion{bold}$v$}_{1}\) and \( \mbox{\mathversion{bold}$v$}_{2}\)
and \( M\geq 0\) such that
\begin{align*}
\| \mbox{\mathversion{bold}$v$}_{1}(t)
-\mbox{\mathversion{bold}$v$}_{2}(t)\|_{3}
\leq
C_{0}M
\end{align*}
holds for all \( t>0\).
Then, there exists \( C_{\ast \ast}>0\) depending on \( C_{0}\) and
\( \| \mbox{\mathversion{bold}$v$}_{0,1}\|_{3}+ \| \mbox{\mathversion{bold}$v$}_{0,2}\|_{3}\)
such that
\begin{align*}
\inf_{\theta\in \mathbf{R}}\| \exp\{ {\rm i}\theta \}q_{1}(t)
-q_{2}(t)\|_{2}
\leq
C_{\ast \ast}M
\end{align*}
holds for all \( t>0\). Here, \( C_{\ast \ast}>0 \)
is non-decreasing with respect to 
\( C_{0}>0 \) and \(  \| \mbox{\mathversion{bold}$v$}_{0,1}\|_{3}+ \| \mbox{\mathversion{bold}$v$}_{0,2}\|_{3}\).
\label{lyaporb}
\end{pr}
Proposition \ref{lyaporb} with \( M\geq 0\) being small implies that 
Lyapunov stability estimates for problem (\ref{slant})
translates to orbital stability estimates for 
problem (\ref{plane}) through the 
generalized Hasimoto transformation.


\subsubsection{Final Step of the Proof of Theorem \ref{stability}}

We combine all of the results of the preceding sections
to finish the proof of Theorem \ref{stability}.

Let \( q_{R}(t)=-\frac{1}{R}\exp\{-\frac{{\rm i}t}{2R^{2}}\} \)
with \( R> L/\pi \) be the plane wave solution of problem (\ref{plane})
and \( \phi_{0}\in H^{2}(I_{L})\) satisfy the \( 0\)-th order 
compatibility condition for (\ref{plane}).
Let \( \mbox{\mathversion{bold}$v$}^{R} \) be the arc-shaped solution 
of problem (\ref{slant}) which corresponds to \( q_{R}\) through
the inverse generalized Hasimoto transformation as shown in 
Section \ref{hasiplane}. Recall that 
\( \mbox{\mathversion{bold}$v$}^{R}=\mbox{\mathversion{bold}$v$}^{R}(s)\)
is a stationary solution. Furthermore, 
let \( \mbox{\mathversion{bold}$v$}_{0}\in H^{3}(I_{L})\) 
be the initial datum corresponding to
\( q_{R}(0)+ \phi_{0} \) and set 
\( \mbox{\mathversion{bold}$\varphi$}_{0}:= \mbox{\mathversion{bold}$v$}_{0}-\mbox{\mathversion{bold}$v$}^{R} \). 
Finally, let \( \mbox{\mathversion{bold}$v$}\in C\big( [0,\infty);H^{3}(I_{L})\big)\cap C^{1}\big(  [0,\infty);H^{1}(I_{L})\big) \) be the solution 
of problem (\ref{slant}) with initial datum
\( \mbox{\mathversion{bold}$v$}_{0} \) and boundary datum
\( \mbox{\mathversion{bold}$b$}=\mbox{\mathversion{bold}$v$}^{R}(L)\)
and set
\( \mbox{\mathversion{bold}$\varphi$}:= \mbox{\mathversion{bold}$v$}-\mbox{\mathversion{bold}$v$}^{R}\).
From the results of Section \ref{inverse hasimoto},
we have
\begin{align}
\| \mbox{\mathversion{bold}$\varphi$}_{0}\|_{3}=
\| \mbox{\mathversion{bold}$v$}_{0} 
-\mbox{\mathversion{bold}$v$}^{R}\|_{3}
\leq
C\big(
\|\phi_{0}\|_{2} + \| \phi_{0}\|_{2}^{3}
\big),
\label{estest}
\end{align}
where \( C>0\) depends on \( \|q_{R}(0)\|_{2}+\| \phi_{0}\|_{2} \)
and is non-decreasing with respect to 
\( \|q_{R}(0)\|_{2}+\| \phi_{0}\|_{2} \). Note that \( q_{R} \)
doesn't depend on \( s\), so that \( C\) essentially depends on 
\( R\) and \( \| \phi_{0}\|_{2}\), and is non-decreasing with respect to
\( \| \phi_{0}\|_{2}\).
From the above estimate and Proposition \ref{arcstable}, there exists 
\( C_{3}>0\), which depends on \( L\), \( R\), and 
\( \| \phi_{0}\|_{2} \), such that
\begin{align*}
\| \mbox{\mathversion{bold}$v$}(t)-\mbox{\mathversion{bold}$v$}^{R}\|_{3}
\leq
C_{3}\big(
\|\phi_{0}\|_{2} + \| \phi_{0}\|^{9}_{2}
\big)
\end{align*}
holds for all \( t>0\). Proposition \ref{lyaporb} further 
shows that there exists \(C_{4} >0\) depending on 
\( L\), \(R\), \( \| \phi_{0}\|_{2} \), and
\( \| \mbox{\mathversion{bold}$v$}^{R}\|_{3}+\| \mbox{\mathversion{bold}$v$}_{0}\|_{3}\) such that
\begin{align*}
\inf_{\theta \in \mathbf{R}}\| \exp\{{\rm i}\theta\}q(t)-
q_{R}(t)\|_{2}
\leq
C_{4}\big(
\|\phi_{0}\|_{2} + \| \phi_{0}\|^{9}_{2}
\big)
\end{align*}
holds for all \( t>0\), and
\( C_{4}>0\) is non-decreasing with respect to 
\( \| \phi_{0}\|_{2}\) and \( \| \mbox{\mathversion{bold}$v$}^{R}\|_{3}+\| \mbox{\mathversion{bold}$v$}_{0}\|_{3}\). Again, note that
\( \| \mbox{\mathversion{bold}$v$}^{R}\|_{3}\) is essentially a constant 
depending on \( R\). From inequality (\ref{estest})
we see that
\begin{align*}
\| \mbox{\mathversion{bold}$v$}_{0}\|_{3}
\leq C(1+\| \phi_{0}\|_{2}+\| \phi_{0}\|_{2}^{3}),
\end{align*}
where \( C>0\) depends on \( R \) and \( \|\phi_{0}\|_{2}\), and is 
non-decreasing with respect to \( \|\phi_{0}\|_{2}\).
Hence, \( C_{4}>0\) can be chosen uniformly with respect to \( \phi_{0}\)
satisfying \( \| \phi_{0}\|_{2} \leq 1\).

Finally, for an arbitrary \( \varepsilon >0\), 
choose \( 0<\delta<1 \) small enough such that
\( C_{4}\big( \delta + \delta^{9}\big)<\varepsilon \) holds.
Then, if \( \|\phi_{0}\|_{2}\leq \delta \), 
\begin{align*}
\inf_{\theta \in \mathbf{R}}\| \exp\{{\rm i}\theta\}q(t)-
q_{R}(t)\|_{2}
<
\varepsilon
\end{align*}
holds for all \( t>0\), and this finishes the proof of Theorem \ref{stability}.

\hfill \( \Box \)

\section*{Acknowledgement}

This work was supported in part by JSPS Grant-in-Aid for Early-Career Scientists
grant number 20K14348.


\begin{thebibliography}{99}

\bibitem{23}
M. Aiki,
Motion of a Vortex Filament on a Slanted Plane,
{\it J. Differential Equations}, {\bf 263} (2017), no. 10, pp. 6885--6915.

\bibitem{27}
M. Aiki,
Long-time behavior of an Arc-shaped Vortex Filament,
in preparation.


\bibitem{11}
M. Aiki and T. Iguchi,
Motion of a vortex Filament in the half space,
{\it Nonlinear Anal.}, {\bf 75} (2012), pp. 5180--5185.





\bibitem{21}
R.J. Arms and F.R. Hama, 
Localized-induction concept on a curved vortex and motion of an elliptic vortex ring,
{\it Phys. Fluids}, {\bf 8} (1965), no.4, pp. 553--559.

\bibitem{15}
V. Banica and L. Vega,
On the Stability of a Singular Vortex Dynamics,
{\it Comm. Math. Phys.}, {\bf 286} (2009), pp. 593--627.

\bibitem{16}
V. Banica and L. Vega,
Scattering for 1D cubic NLS and singular vortex dynamics,
{\it J. Eur. Math. Soc.}, {\bf 14} (2012), pp. 209--253.

\bibitem{25}
V. Banica and L. Vega,
Stability of the self-similar dynamics of a vortex filament,
{\it Arch. Ration. Mech. Anal.}, {\bf 210} (2013), no. 3, pp. 673--712.


\bibitem{57}
J. L. Bona, S. M. Sun, and B. Y. Zhang,
Nonhomogeneous boundary-value problems for one-dimensional
nonlinear Schr\"odinger equations,
{\it J. Math. Pures Appl. (9)}, {\bf 109} (2018), pp. 1--66.



\bibitem{41}
J. Bourgain,
Fourier transform restriction phenomena for certain lattice subsets
and applications to nonlinear evolution equations I. Schr\"odinger equations,
{\it Geom. Funct. Anal.}, {\bf 3} (1993), no. 2, pp. 107--156.


\bibitem{63}
J. Bourgain,
Periodic nonlinear Schr\"odinger equation and invariant measures,
{\it Comm. Math. Phys.}, {\bf 166} (1994), no. 1, pp. 1--26.




\bibitem{62}
T. Cazenave and P.-L. Lions,
Orbital stability of standing waves for some nonlinear Schr\"odinger equations,
{\it Comm. Math. Phys.}, {\bf 85} (1982), no. 4, pp. 549--561.






\bibitem{26}
N. Chang, J. Shatah, and K. Uhlenbeck,
Schr\"odinger maps,
{\it Comm. Pure Appl. Math.}, {\bf 53} (2000), no. 5, pp. 590--602.


\bibitem{20}
L.S. Da Rios, 
Sul moto d'un liquido indefinito con un filetto vorticoso di forma qualunque (in Italian),
{\it Rend. Circ. Mat. Palermo}, {\bf 22} (1906), no. 3, pp. 117--135.

\bibitem{31}
E. Faou, L. Gauckler, and C. Lubich,
Sobolev stability of plane wave solutions to the cubic nonlinear Schr\"odinger equation
on a torus,
{\it Comm. Partial Differential Equations}, {\bf 38} (2013), no. 7, pp. 1123--1140.


\bibitem{52}
A. S. Fokas and A. R. Its,
The nonlinear Schr\"odinger equation on the interval,
{\it J. Phys. A}, {\bf 37} (2004), no. 23, pp. 6091--6114.


\bibitem{53}
A. S. Fokas, A. A. Himonas, and D. Mantzavions,
The nonlinear Schr\"odinger equation on the half-line,
{\it Trans. Amer. Math. Soc.}, {\bf 369} (2017), no. 1, pp. 681--709.


\bibitem{54}
A. S. Fokas, A. R. Its, and L. Y. Sung,
The nonlinear Schr\"odinger equation on the half-line,
{\it Nonlinearity}, {\bf 18} (2005), no. 4, pp. 1771--1822.








\bibitem{29}
T. Gallay and M. H\u{a}r\u{a}gu\c{s},
Stability of small periodic waves for the nonlinear Schr\"odinger equation,
{\it J. Differential Equations}, {\bf 234} (2007), no. 2, pp. 544--581.

\bibitem{30}
T. Gallay and M. H\u{a}r\u{a}gu\c{s},
Orbital stability of periodic waves for the nonlinear Schr\"odinger equation,
{\it J. Dynam. Differential Equations}, {\bf 19} (2007), no. 4, pp. 825--865.

\bibitem{60}
M. Grillakis, J. Shatah, and W. Strauss,
Stability theory of solitary waves in the presence of symmetry. I.
{\it J. Funct. Anal.}, {\bf 74} (1987), no. 1, pp. 160--197.

\bibitem{61}
M. Grillakis, J. Shatah, and W. Strauss,
Stability theory of solitary waves in the presence of symmetry. II.
{\it J. Funct. Anal.}, {\bf 94} (1990), no. 2, pp. 308--348.







\bibitem{17}
S. Guti\'errez, J. Rivas, and L. Vega,
Formation of Singularities and Self-Similar Vortex Motion Under the Localized Induction Approximation,
{\it Comm. Partial Differential Equations}, {\bf 28} (2003), no. 5 and 6, pp. 927--968.

\bibitem{14}
H. Hasimoto, 
A soliton on a vortex filament, 
{\it J. Fluid Mech.}, {\bf 51} (1972), no. 3, pp. 477--485. 


\bibitem{51}
J. Holmer, 
The initial-boundary-value problem for the 1D nonlinear Schr\"odinger equation on the half-line,
{\it Differential Integral Equations}, {\bf 18} (2005), no. 6, pp. 647--668.


\bibitem{55}
J. Lenells and A. S. Fokas,
The nonlinear Schr\"odinger equation with t-periodic data: I. Exact results,
{\it  Proc. A.}, {\bf 471} (2015), no. 2181, 20140925, 22 pages.

\bibitem{56}
J. Lenells and A. S. Fokas,
The nonlinear Schr\"odinger equation with t-periodic data: II. Perturbative results,
{\it  Proc. A.}, {\bf 471} (2015), no. 2181, 20140926, 22 pages.











\bibitem{18}
T. Kato,
Nonstationary Flows of Viscous and Ideal Fluids in \( \mathbf{R}^{3}\),
{\it J. Functional Analysis}, {\bf 9} (1972), pp. 296--305.

\bibitem{24}
A. Nahmod, J. Shatah, L. Vega, and C. Zeng,
Schr\"odinger maps and their associated frame systems,
{\it Int. Math. Res. Not.}, {\bf 21} (2007), 29 pages.

\bibitem{12}
N. Koiso, 
The Vortex Filament Equation and a Semilinear Schr\"odinger Equation in a Hermitian Symmetric Space, 
{\it Osaka J. Math.}, {\bf 34} (1997), no. 1, pp. 199--214. 

\bibitem{22}
Y. Murakami, H. Takahashi, Y. Ukita, and S. Fujiwara,
On the vibration of a vortex filament (in Japanese),
{\it Applied Physics Colloquium} (in Japanese), {\bf 6} (1937), pp. 1--5.


\bibitem{10}
T. Nishiyama,
Existence of a solution to the Mixed Problem for a 
Vortex Filament Equation with an External Flow Term,
{\it J. Math. Scie. Univ. Tokyo}, {\bf 7} (2000), no. 1, pp. 35--55.

\bibitem{5}
T. Nishiyama and A. Tani, 
Initial and Initial-Boundary Value Problems for a Vortex Filament with or without Axial Flow, 
{\it SIAM J. Math. Anal.}, {\bf 27} (1996), no. 4, pp. 1015--1023. 

\bibitem{13}
T. Nishiyama and A. Tani, 
Solvability of the localized induction equation for vortex motion, 
{\it Comm. Math. Phys.}, {\bf 162} (1994), no. 3, pp. 433--445. 

\bibitem{1}
J. B. Rauch and F. J. Massey, 
Differentiability of solutions to hyperbolic initial-boundary value problems,
{\it Trans. Amer. Math. Soc.}, {\bf 189} (1974), pp. 303--318.


\bibitem{44}
G. Rowlands,
On the Stability of Solutions of the Non-linear Schr\"odinger Equation,
{\it J. Inst. Maths. Applics.}, {\bf 13} (1974), pp. 367--377.





\bibitem{2}
V. A. Solonnikov,
An initial-boundary value problem for a 
Stokes system that arises in the study of a problem with a free boundary,
{\it Proc. Steklov Inst. Math.}, {\bf 3} (1991), pp. 191--239.



\bibitem{42}
B. Wilson,
Sobolev stability of plane wave solutions to the nonlinear Schr\"odinger equation,
{\it Comm. Partial Differential Equations}, {\bf 40} (2015), no. 8, pp.1521--1542.




\bibitem{40}
V. E. Zakharov and A. B. Shabat,
Exact theory of two-dimensional self-focusing and one-dimensional self-modulation of waves
in nonlinear media,
{\it Soviet Physics JETP}, {\bf 34} (1972), no.1, pp. 62--69.



\bibitem{28}
P. E. Zhidkov,
Korteweg-de Vries and nonlinear Schr\"odinger equations: qualitative theory,
{\it Lecture Notes in Mathematics}, vol. 1756, Springer-Verlag, Berlin, (2001)


\end{thebibliography}


\end{document}